\documentclass{amsart}                     % onecolumn (standard format)

\usepackage{graphicx}
\usepackage[english]{babel}
\usepackage{pdfpages}
\usepackage{amssymb,amsmath,latexsym,amsfonts}
\usepackage{hyperref}
\usepackage{functan}
\usepackage{subfig}
\usepackage{enumitem}
\usepackage{float}
\usepackage{bbm}
\usepackage[ruled,vlined]{algorithm2e}

\usepackage{pstricks,pst-plot,pst-node,pst-xkey,xkeyval,pst-3dplot}
\definecolor{cyan}{cmyk}{1,0,0,0}
\definecolor{magenta}{cmyk}{0,1,0,0}
\definecolor{yellow}{cmyk}{0,0,1,0}
\definecolor{black}{cmyk}{0,0,0,1}

\definecolor{white}{cmyk}{0,0,0,0}
\definecolor{gray}{cmyk}{0,0,0,0.5}

\definecolor{red}{cmyk}{0,1,1,0}
\definecolor{green}{cmyk}{1,0,1,0}
\definecolor{blue}{cmyk}{1,1,0,0}

\definecolor{palered}{cmyk}{0,0.25,0.25,0}
\definecolor{palegreen}{cmyk}{0.25,0,0.25,0}
\definecolor{paleblue}{cmyk}{0.25,0.25,0,0}
\definecolor{palecyan}{cmyk}{0.25,0,0,0}
\definecolor{palemagenta}{cmyk}{0,0.25,0,0}
\definecolor{paleyellow}{cmyk}{0,0,0.25,0}
\definecolor{palegray}{cmyk}{0,0,0,0.05}
\definecolor{lightred}{cmyk}{0,0.5,0.5,0}
\definecolor{lightgreen}{cmyk}{0.5,0,0.5,0}
\definecolor{lightblue}{cmyk}{0.5,0.5,0,0}
\definecolor{lightcyan}{cmyk}{0.5,0,0,0}
\definecolor{lightmagenta}{cmyk}{0,0.5,0,0}
\definecolor{lightyellow}{cmyk}{0,0,0.5,0}
\definecolor{lightgray}{cmyk}{0,0,0,0.1}
\definecolor{mediumred}{cmyk}{0,0.75,0.75,0}
\definecolor{mediumgreen}{cmyk}{0.75,0,0.75,0}
\definecolor{mediumblue}{cmyk}{0.75,0.75,0,0}
\definecolor{mediumcyan}{cmyk}{0.75,0,0,0}
\definecolor{mediummagenta}{cmyk}{0,0.75,0,0}
\definecolor{mediumyellow}{cmyk}{0,0,0.75,0}
\definecolor{mediumgray}{cmyk}{0,0,0,0.25}
\definecolor{heavyred}{cmyk}{0,1,1,0.25}
\definecolor{heavygreen}{cmyk}{1,0,1,0.25}
\definecolor{heavyblue}{cmyk}{1,1,0,0.25}
\definecolor{heavycyan}{cmyk}{1,0,0,0.25}
\definecolor{heavymagenta}{cmyk}{0,1,0,0.25}
\definecolor{lightolive}{cmyk}{0,0,1,0.25}
\definecolor{heavygray}{cmyk}{0,0,0,0.75}
\definecolor{deepred}{cmyk}{0,1,1,0.5}
\definecolor{deepgreen}{cmyk}{1,0,1,0.5}
\definecolor{deepblue}{cmyk}{1,1,0,0.5}
\definecolor{deepcyan}{cmyk}{1,0,0,0.5}
\definecolor{deepmagenta}{cmyk}{0,1,0,0.5}
\definecolor{olive}{cmyk}{0,0,1,0.5}
\definecolor{deepgray}{cmyk}{0,0,0,0.9}
\definecolor{darkred}{cmyk}{0,1,1,0.75}
\definecolor{darkgreen}{cmyk}{1,0,1,0.75}
\definecolor{darkblue}{cmyk}{1,1,0,0.75}
\definecolor{darkcyan}{cmyk}{1,0,0,0.75}
\definecolor{darkmagenta}{cmyk}{0,1,0,0.75}
\definecolor{darkolive}{cmyk}{0,0,1,0.75}
\definecolor{darkgray}{cmyk}{0,0,0,0.95}
\definecolor{orange}{cmyk}{0,0.5,1,0}
\definecolor{fuchsia}{cmyk}{0,1,0.5,0}
\definecolor{chartreuse}{cmyk}{0.5,0,1,0}
\definecolor{springgreen}{cmyk}{1,0,0.5,0}
\definecolor{purple}{cmyk}{0.5,1,0,0}
\definecolor{royalblue}{cmyk}{1,0.5,0,0}
\definecolor{salmon}{cmyk}{0,0.5,0.5,0}
\definecolor{brown}{cmyk}{0,1,1,0.5}
\definecolor{darkbrown}{cmyk}{0,1,1,0.75}
\definecolor{pink}{cmyk}{0,0.25,0,0}
\definecolor{palegrey}{cmyk}{0,0,0,0.05}
\definecolor{lightgrey}{cmyk}{0,0,0,0.1}
\definecolor{mediumgrey}{cmyk}{0,0,0,0.25}
\definecolor{grey}{cmyk}{0,0,0,0.5}
\definecolor{heavygrey}{cmyk}{0,0,0,0.5}
\definecolor{deepgrey}{cmyk}{0,0,0,0.9}
\definecolor{darkgrey}{cmyk}{0,0,0,0.95}

\newcommand{\RR}{{\mathbb R}}
\newcommand{\NN}{{\mathbb N}}
\newcommand{\VV}{{\mathbb V}}
\newcommand{\TT}{{\mathbb S}}
\newcommand{\X}{{\mathbf X}}

\def\U{{\mathbf U}}
\newcommand{\W}{{\mathbf W}}
\newcommand{\UU}{{\mathbb U}}
\newcommand{\WW}{{\mathbb W}}
\renewcommand{\u}{{\mathbf {u}}}
\newcommand{\x}{\mathbf x}
\newcommand{\XX}{{\mathbb X}}
\newcommand{\state}{x}
\newcommand{\control}{u}
\newcommand{\w}{\mathbf w}
\newcommand{\scenario}{\omega}
\newcommand{\dynamics}{F}
\newcommand{\ic}{\xi}
\newcommand{\constraints}{g}
\newcommand{\fcost}{\theta}
\newcommand{\threshold}{c}
\newcommand{\feedback}{\vartheta}
\newcommand{\bfthreshold}{{\bf\threshold}}

\newcommand{\defegal}{:=}
\newcommand{\pareto}{\mathcal P}
\newcommand{\wpareto}{{\mathcal P}_w}

\def\msy{\rm{MSY}}
\def\bfone{\mathbf 1}

\newcommand{\fv}[2]{\vartheta_{#1}^{#2}} %funcion valor (problema original)
\newcommand{\titf}[2]{[\![{#1}\!:\!{#2}]\!]}%{\{#1,\ldots, #2\}} %intervalo de tiempo
\newcommand{\sd}[2]{D_{#1}^{#2}} % sistema dinámico asociado a n y \ic
\newcommand{\mconstraints}[2]{\text{I}_{#1}^{#2}} %sistema de las restricciones mixtas
\newcommand{\econstraints}[2]{\text{E}_{#1}^{#2}} %sistema de las restricciones mixtas
\newcommand{\gcost}[1]{\mathcal{J}_{#1}} %costo generalizado
\newcommand{\aux}[2]{\mathcal{W}_{#1}\left({#2}\right)} % función auxiliary w
\newcommand{\R}[2]{R_{#1}^{#2}} % función generalizada de restricciones
\newcommand{\pg}[1]{{#1}}
\newcommand{\new}[1]{{#1}}

\newtheorem{theorem}{Theorem}[section]

\newtheorem{proposition}[theorem]{Proposition}

\theoremstyle{definition}
\newtheorem{definition}[theorem]{Definition}

\theoremstyle{remark}
\newtheorem{remark}[theorem]{Remark}
\begin{document}

\title[On the set of robust sustainable thresholds ]{On the set of robust sustainable thresholds for uncertain control systems}
%\subtitle{Some characterizations via optimal control theory}

%\titlerunning{the set of robust sustainable thresholds via a level-set approach}        

\author{Pedro Gajardo}
\address[Pedro Gajardo]{Departamento de Matem\'atica, Universidad T\'ecnica Federico Santa Mar\'ia, Av. Espa\~na 1680, Valpara\'iso, Chile}
\email{pedro.gajardo@usm.cl}

\author{Cristopher Hermosilla}
\address[Cristopher Hermosilla]{Departamento de Matem\'atica, Universidad T\'ecnica Federico Santa Mar\'ia, Av. Espa\~na 1680, Valpara\'iso, Chile}
\email{cristopher.hermosill@usm.cl}

\author{Athena Picarelli}
\address[Athena Picarelli]{Dipartimento di Scienze Economiche, Universit\`a di Verona, Polo Santa Marta, Verona, Italy}
\email{athena.picarelli@univr.it}
\thanks{This work was supported by FONDECYT grants N {1200355 (P. Gajardo, C. Hermosilla)} and N 11190456 (C. Hermosilla), both ANID-Chile programs.\\
{\bf Authors contribution:} the authors collaborated to the design of the research, to the analysis of the
results and to the writing of the manuscript.}

%\date{\today}
% The correct dates will be entered by the editor



\maketitle

\begin{abstract}
In natural resource management, or more generally in the study of sustainability issues, often the objective is to maintain the state of a given system within a desirable configuration, typically established in terms of standards or thresholds. For instance, in fisheries management, the procedure of designing policies may  include keeping the spawning stock biomass over a critical threshold and  also ensuring minimal catches. Given a controlled dynamical system in discrete-time, representing the evolution of some natural resources under the action of controls and uncertainties, and an initial endowment of the resources, the aim of this paper is to characterize the set of robust sustainable thresholds, that is,  the thresholds for which there exists some control path, along with its corresponding  state trajectory, satisfying for all possible uncertainty scenarios, prescribed mixed constraints parametrized by such thresholds. This set provides useful information to users and decision-makers, illustrating the trade-offs between  constraints and it is strongly related to the robust viability, one of the key concepts in viability theory, discipline that study the consistency between a controlled dynamical system and given constraints. Specifically, we are concerned with characterizing the weak and strong Pareto fronts of the set of robust sustainable thresholds, providing a practical method for computing such objects based on optimal control theory and a level-set approach. A numerical example, relying on renewable resource management, is shown to demonstrate the proposed method.

\end{abstract}

%\tableofcontents
\section{Introduction}
In natural resource management or broadly  in the study of sustainability issues,  to determine biological, ecological or social constraints to fulfill throughout time emerges as a crucial issue. Mathematically speaking, one of the objectives of decision-makers can be seen as to maintain the state of a given system within a desirable configuration, typically established in terms of constraints  parametrized by standards or thresholds.  For instance, in fisheries management, the procedure of designing policies may  include keeping the spawning stock biomass over a critical threshold and  also ensuring minimal catches. In this example, the first requirement is associated to the sustainability of the resource and the second to economical, social or food security issues.  The focus on constraints   is well adapted  also to address biodiversity conservation problems or  the climate change issue. In this framework, reference points not to exceed for  biological, ecological, economic, or social indicators stand for sustainable management objectives. As examples of this approach, one can mention the concept of Safe Minimum Standards (SMS)  \cite{MargolisNaevdal:2008} where tipping thresholds and risky areas are introduced, or the Tolerable Windows Approach (TWA) \cite{Bruckneretal:1999}, based on safe boundaries and feasibility regions. If the constraints induced by  thresholds or tipping points have to be satisfied over time, such  problems related to sustainability can be formulated into the mathematical framework of viability theory  \cite{Aubin:1990,DLD,Aubin:2011}. Indeed, this approach has been  applied by numerous authors to the sustainable management of renewable resources \cite{BeneDoyen:2000,Beneetal:2001,Krawczyketal:2013,Pereauetal:2012,Pereauetal:2018,Doyenetal:2017,sp1,sp2}, as recently is reviewed in \cite{SchuhbauerSumaila:2016,zaccour}.

Given a controlled dynamical system in discrete-time, representing for instance  the evolution of some natural resources under the action of controls and uncertainties, and an initial endowment of the resources, or more generally an initial state, the aim of this paper is to characterize the \emph{set of robust sustainable thresholds}, composed by the collection of all possible thresholds  for which there exists a control strategy, along with its corresponding state trajectory, satisfying for all possible uncertainty scenarios, prescribed mixed constraints parametrized by such thresholds. In the deterministic framework, this set, called the \emph{set of sustainable threshold or standards}, has been studied recently in \cite{BGV2018,GOR2018,martinet2011,MGDR} and characterized in \cite{DG2018,GajHer19}. In  \cite{GajHer19}, the starting point of the current work,  a characterization of the strong and weak Pareto front for this set is provided.

\if{
The aim of this paper is to characterize the set of robust sustainable thresholds for an uncertain discrete-time control system. Given a controlled dynamical system, the set of sustainable thresholds is the collection of all possible thresholds for which a given initial condition is viable (or sustainable) throughout time \cite{BGV2018,GOR2018,martinet2011,MGDR}. The latter means that some control, along with its corresponding controlled trajectory, satisfies prescribed mixed constraints parametrized by thresholds. For deterministic systems, a characterization of the set of sustainable thresholds, together with its strong and weak Pareto front, has been recently proposed in \cite{GajHer19}.
}\fi

We study here uncertain control systems motivated by practical applications, where the limited knowledge about the phenomena that influence the system evolution, and the role of uncertainty and its quantification becomes particularly relevant. 
%In many practical applications, due to the limited knowledge about the phenomena that influence the system evolution, the role of uncertainty and its quantification becomes particularly relevant. 
In environmental management problems uncertainty typically affects the model as a result of environmental changes that influence natural mechanisms (see, for instance, \cite{LanEngSae03,OlsSan00}). Moreover, uncertainty can also be used to reflect the possibility of measurement errors. In presence of an uncertainty, constraints can be considered in different ways. Typical examples are constraints imposed in probability, expectation and sure (or almost sure) path-wise constraints. In this paper we deal with the last ones: given a set of scenarios reflecting the possible future states of the world, the set of robust sustainable thresholds defines the collection of thresholds that are sustainable under any scenario.  Accordingly, the set of robust sustainable thresholds provides a good picture of the current state of a system in terms of its maintainable sustainability under any possible occurrence, \cite{FreeKoko96}. In particular, for a given initial state, a small set of robust sustainable thresholds means that there exist some circumstances for which there is a limited possibility to operate in a sustainable way.

We point out that the study of robustness in control theory arises in several frameworks. Among the main motivations for this work, we mention sustainable management problems, bio-economic modeling and robust viability, see for instance \cite{DoyDeLaFerrPell07, doyen2003sustainability, TBDG04, DoyePere09,sepulveda2018}. 

In this paper, we first obtain   characterizations of the strong and weak Pareto fronts of the set of robust sustainable thresholds,  and then we use such a characterization to provide a practical method for their approximation. To achieve these goals, we make use of optimal control tools.  In particular, in our main theoretical results we prove that it is possible to describe the strong Pareto front of robust sustainable thresholds by solving a finite number of optimal control problems and that the weak Pareto front corresponds to the zero level set of the value function associated to a suitable unconstrained optimal control problem.
We then use this characterization and the dynamic programming principle to provide an implementable scheme for approximating the set of robust sustainable thresholds and its weak Pareto front. 

We remark that for the characterization of the weak Pareto front we are inspired by the so-called \emph{level-set approach}. Introduced in \cite{OS88} to describe the propagation of fronts in continuous time, the idea at the basis of this approach is to link the set of interest (the set of robust sustainable thresholds in our case) to the  level set of a suitable auxiliary function which can be numerically approximated. In the deterministic continuous time framework this technique has been successfully applied in \cite{BokaForcZida10,MR2151560} to characterize the set of admissible initial condition in presence of (pure) state constraints. The approach has then been extended to the stochastic case  in \cite{BPZ15}, where almost sure path-wise state constraints are taken into account.
{The use of this technique leads us to work with maximin problems as those considered in \cite{MR3090148} and \cite{EsteAschStrei20} to describe the boundaries of admissible sets for continuous time systems  respectively in the deterministic and uncontrolled robust framework. However, we stress that in all the aforementioned works the focus is on the characterization and approximation of the set of sustainable initial conditions for a given threshold, the so-called \emph{viability kernel}, while we are interested to determine the set of thresholds that are sustainable once the initial condition is fixed.}

This manuscript is organized as follows. In Section \ref{sec:preliminaries} we present some preliminary concepts on discrete-time systems introducing the set of robust sustainable thresholds. In this section we establish the standing assumptions for the rest of the paper. The links between apropiate optimal control problems and the set of robust sustainable threshold are established in Section \ref{sec:SRST_OC}, where a characterization of its strong Pareto front is provided.  In Section \ref{weakpf} we characterize the weak Pareto front of the set of robust sustainable thresholds providing a method for computing this front, based on the dynamic programming principle. Finally, in Section \ref{sec:simulations}, we illustrate the method introduced in Section \ref{weakpf} with an example based on renewable resource management.

\section{Preliminaries on discrete-time control systems}\label{sec:preliminaries}

Given a finite time horizon $N \in \NN\setminus\{0\}$, an initial state $\ic \in \X$, a finite sequence of controls $\u=(\control_k)_{k=0}^{N}$  and a scenario $\w=(\scenario_k)_{k=0}^{N}$, we consider the uncertain discrete-time control system:

	\begin{equation}
	\label{eq:system}\tag{$\sd{\ic}{\u}(\w)$}
	\state_{k+1}=\dynamics_k(\state_k,\control_k,\scenario_k),\quad 
	k\in\titf{0}{N},\quad 
	\state_0=\ic.	
  \end{equation}
  
The data of the problem comprise the \pg{dynamics} $\dynamics_0,\ldots,\dynamics_N:\X\times\U\times\W\to\X$, the state space $\X$  (a vector space), the control space $\U$ and  the scenarios' space $\W$. Here, $\titf{p}{q}$ stands for the collection of all integers between $p$ and $q$ (inclusive).

The collection of all possible controls is given by:
	$$\UU\defegal\left\{\u=(\control_k)_{k=0}^{N}\middle|\ \control_0,\ldots,\control_{N-1} \in \U \right\}\cong \U^{N+1}.$$
The possible scenarios are assumed to vary with respect to time in the sense that for any $k\in\titf{0}{N}$ there is $\Omega_k\subseteq\W$ for which $\scenario_k\in\Omega_k$. Consequently, the collection of all possible scenarios is then given by
	$$\WW\defegal\left\{\w=(\scenario_k)_{k=0}^{N}\middle|\ \scenario_k\in\Omega_k,\quad\forall k\in\titf{0}{N} \right\}\cong\prod_{k=0}^{N}\Omega_k.$$

A solution of the uncertain control system \eqref{eq:system} associated with a control $\u\in\UU$ and a scenario $\w\in\WW$ is an element of the space
$$\XX\defegal\left\{\x=(\state_k)_{k=0}^{N+1}\middle|\ \state_0,\ldots,\state_{N+1}\in\X\right\}\cong\X^{N+2},$$
that satisfies the initial time condition $\state_0=\ic$.

Consequently, a solution of \eqref{eq:system}, which is uniquely determined by the control $\u$, a scenario $\w$ and initial state $\ic$, is denoted in the sequel by $\x_{\ic}^\w(\u)$ to emphasize its dependence on the initial data of the problem (control, scenario and state). In Section \ref{sec:SRST_OC} we establish links between appropriate optimal control problems and the set of robust sustainable threshold, characterizing its strong Pareto front.

\subsection{Constraints and sustainable thresholds}\label{sec:mc}
In many real applications, the outputs and inputs of dynamical systems such as \eqref{eq:system} are restricted to a prescribed set, that may reflect {biological, physical, economical or social restrictions}. The uncertain control system considered in this work allows to take different probabilistic interpretations  on how these constraints are satisfied. In this work we are mainly concerned with a robust approach, which means that the set of restrictions considered must be satisfied by the control $\u\in\UU$ together with its corresponding trajectories $\x^\w_{\ic}(\u)=(\state_k)_{k=0}^{N+1}$, for  any possible scenario $\w\in\WW$. To be more precise, we consider the so-called \emph{mixed-constraints} that can be represented as the level-set of a given constraint mappings $\constraints^0,\ldots,\constraints^{N}\colon\X\times\U\to\RR^m$

%\note{We may also consider that the constraints depends on the scenarios:$$\constraints^k(\state_k,\control_k,\scenario_k)\geq c ,  \quad\forall ~k\in\titf{0}{N}$$ and/or there is an end-point constraint that doesn't depend on the control: $$\constraints(N+1,\state_{N+1},\scenario_{N+1})\geq c.$$}

	\begin{equation}\label{eq:constraints}\tag{$\mconstraints{}{\threshold}$}
\constraints^k(\state_k,\control_k)\geq c ,  \quad\forall ~k\in\titf{0}{N}.
\end{equation}

In this work, we also consider that trajectories are forced to satisfy an end-point constraint, which can be represented as the level-set of a given constraint mapping $\fcost\colon\X\to\RR^m$
	\begin{equation}\label{eq:end_point}\tag{$\econstraints{}{\threshold}$}
		\fcost(\state_{N+1})\geq c
	\end{equation}

All the constraints are determined by a parameter $c\in \RR^m$, which is a given vector of thresholds. The focus of our work is on this parameter rather than on the initial conditions as in viability theory \cite{aubin}.

In particular, we are interested in finding, for a given initial state, all possible thresholds $\threshold \in \RR^m$ for which that initial condition is robustly sustainable throughout time. The latter means that some control, along with its corresponding controlled trajectories, satisfy the mixed constraints \eqref{eq:constraints} and the end-point constraints \eqref{eq:end_point} for any possible scenario. The collection of all such thresholds is called the \emph{set of robust sustainable thresholds} and is defined for a given initial condition $\ic \in \X$ as follows:

\begin{equation}\label{eq:thresholds_set}
	\TT(\ic)\defegal\left\{\threshold \in \RR^m \mid \exists \u \in \UU,  \, \u \text{ and } \x_{\ic}^\w(\u)\text{ satisfy   \eqref{eq:constraints} - \eqref{eq:end_point}, for any }\w\in\WW\right\}.
\end{equation}

For a given threshold vector $\threshold \in \RR^m$, the robust viability kernel (\cite{DLD}) associated with the uncertain control system \eqref{eq:system} is given by
$$	\VV(\threshold)\defegal\left\{\ic\in\X \mid \exists \u \in \UU,  \, \u \text{ and } \x_{\ic}^\w(\u)\text{ satisfy  \eqref{eq:constraints} - \eqref{eq:end_point} for any }\w\in\WW\right\}.
$$

It is not difficult to see that, similarly as stated in \cite{DG2018,GajHer19} for the deterministic case, the robust viability kernel and the set of robust sustainable thresholds are related via the equivalence below, which somehow explains the duality between these two objects: for any $\ic \in \X$ and $\threshold \in \RR^m$ we have

\begin{equation}\label{K-SST}
\ic \in \VV(\threshold)\quad\Longleftrightarrow\quad \threshold \in \TT(\ic). 
\end{equation}

\begin{remark}\label{rem:linkdeterministic}
Notice that if $\scenario \in \Omega:=\bigcap_{k=0}^N \Omega_k$ and we consider the constant scenario $\w_{\scenario}=(\scenario_k)_{k=0}^N $ where $~\scenario_k = \scenario ~$ for $~k=0,\ldots,N$, we can define the set of sustainable threshold associated to $\scenario$ by 
\begin{equation*}%\label{eq:thresholds_set}
	\TT^\scenario(\ic)\defegal\left\{\threshold \in \RR^m \mid \exists \u \in \UU,  \, \u \text{ and } \x^{\w}_{\ic}(\u)\text{ satisfy}   \eqref{eq:constraints} - \eqref{eq:end_point} ~\mbox{ for }~  \w=\w_{\scenario} \right\}.
\end{equation*}
This set corresponds to the deterministic set of sustainable thresholds associated to the dynamics $\dynamics_k(\cdot,\cdot,\scenario)$ in \eqref{eq:system} defined in  \cite{DG2018,GajHer19}. From the definition of the set of robust sustainable threshold $\TT(\ic)$ in \eqref{eq:thresholds_set}, it is straightforward to verify that
$$\TT(\ic) \subseteq \hat \TT(\ic) := \bigcap_{w \in \Omega} \TT^\scenario(\ic) .$$

\end{remark}

\begin{remark}
	Other cases that may be worth studying are when the constraints are satisfied only in some scenarios, which can be quantified with a probability measure $\mathbb{P}$ defined on $\WW$. For example, one may be interested in studying the \emph{set of stochastic sustainable thresholds} related  to a confidence level $\beta\in(0,1]$, defined as follows:
$$		\TT^\beta(\ic)\defegal\left\{\threshold \in \RR^m ~\mid~ \exists ~\u \in \UU,  \, \mathbb{P}\left(\w\in\WW\mid\begin{matrix}\constraints^k(\state_k,\control_k)\geq c, \ \forall k\in\titf{0}{N}\\\fcost(\state_{N+1})\geq c\end{matrix}\right)\geq\beta \right\},$$
where $\u=(\control_k)_{k=0}^{N}$ and $\x_{\ic}^\w(\u)=(\state_k)_{k=0}^{N+1}$. In this case, the relation \eqref {K-SST} can also be stated however by replacing the robust viability kernel with the \emph{stochastic viability kernel} (\cite{DLD}): $$		\VV^\beta(c)\defegal\left\{\ic \in \RR^n ~\mid~ \exists ~\u \in \UU,  \, \mathbb{P}\left(\w\in\WW\mid\begin{matrix}\constraints^k(\state_k,\control_k)\geq c, \ \forall k\in\titf{0}{N}\\\fcost(\state_{N+1})\geq c\end{matrix}\right)\geq\beta \right\}.$$
Research involving the stochastic set of sustainable thresholds is beyond the scope of this paper, and we plan to study it elsewhere.
\end{remark}

As in the deterministic case, the importance of the set of robust sustainable thresholds lies in the trade-off between the number of restrictions $m\in\NN$ and the dimension of the state space $\X$. For problems with several state variables, computing $\VV(\threshold)$ may be too expensive or impractical in terms of computational time, even if there are only a couple of restrictions; this is the so-called \emph{curse of dimensionality in dynamic programming}. However, in the same situation (several state variables with few constraints), the computational time of estimating $\TT(\ic)$ can be considerably \pg{lowered} because, essentially, the complexity of computing $\VV(\threshold)$ and $\TT(\ic)$ is the same, but the latter is an object in a lower-dimensional Euclidean space. This fact  makes somewhat the numerical computation of $\TT(\ic)$ more tractable than the one of $\VV(\threshold)$, as we describe in this work. 

The underlying idea of the the set of robust sustainable thresholds $\TT(\ic)$  is to provide a good picture of the current state $\ic$ in terms of the thresholds that can be maintained in a sustainable way through time. A small set $\TT(\ic)$ means that the current state $\ic$ is vulnerable in the sense that the room for maneuvering in terms of sustainability is reduced. In Figure \ref{fig:S}, we illustrate the set of robust sustainable thresholds for two different initial states $\ic$ and $\ic'$, where \eqref{eq:constraints} and \eqref{eq:end_point} consist of only two constraints (i.e., the threshold space is of dimension two). In this illustration, we can see that the state $\ic$ is better than $\ic'$ in the sense that $\TT(\ic') \subset \TT(\ic)$. \\

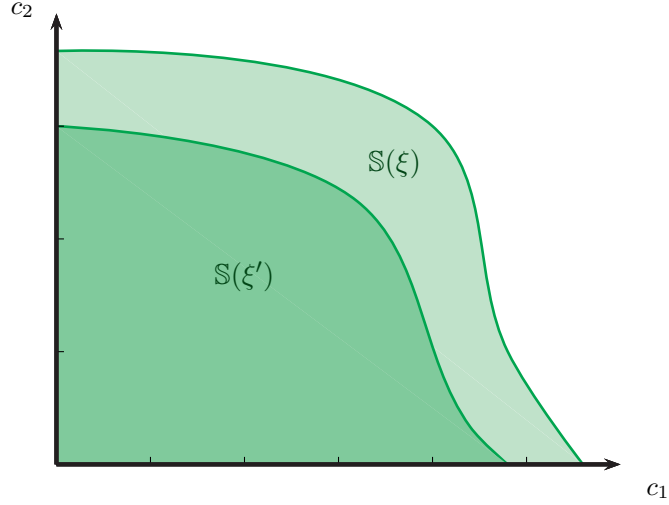
\begin{figure}[h]
\begin{center}
\psset{xunit=.5cm, yunit=.5cm, linewidth=.2pt}

\begin{pspicture}(1,-2)(12,13)

%CURVE  (LARGE)
\pscurve[linewidth=1pt, linecolor=green,fillstyle=solid,fillcolor=palegreen](0,11)(10,9)(12,3)(14,0)
\psline[linewidth=1pt, linecolor=white,fillstyle=solid,fillcolor=palegreen](0,11)(0,0)(14,0)

%CURVE (SMALL)
\pscurve[linewidth=1pt, linecolor=green,fillstyle=solid,fillcolor=lightgreen](0,9)(8,7)(11,1)(12,0)
\psline[linewidth=1pt, linecolor=white,fillstyle=solid,fillcolor=lightgreen](0,9)(0,0)(12,0)
\rput(5,5) {\textcolor{darkgreen}{\large $\TT(\ic')$}}
\rput(9,8) {\textcolor{darkgreen}{\large $\TT(\ic)$}}

% LABELS - TEXT
\rput(-0.9,12.1) {$\threshold_2$}
\rput(16,-0.7) {$\threshold_1$}
\psaxes[linecolor=black,  arrows=->, linewidth=1.5pt,Ox=,Oy=,Dx=20,Dy=2,dx=2.5,dy=3,labels=none,tickstyle=bottom,ticksize=.2](0,0)(15,12)

\end{pspicture}
 \end{center}
 \caption{Sketch of the set of robust sustainable thresholds for two different initial states.}\label{fig:S}
\end{figure}

\subsection{Pareto fronts}

Because of the structure of the constraints \eqref{eq:constraints} and \eqref{eq:end_point}, it is clear that if $\threshold^*\geq c$ (component-wise), then for any $\ic\in\X$, we have

\begin{equation}\label{caractS}
\threshold^*\in\TT(\ic)\quad\Longrightarrow\quad \threshold\in\TT(\ic).\end{equation}

In other words, $\TT(\ic) + \RR^m_- = \TT(\ic)$, and therefore, the set of robust sustainable thresholds can be characterized by its boundary, particularly by its weak Pareto front; \new{See Remark \ref{rem:charact} below}. In this context, let us recall that a given vector $\threshold^*\in\RR^m$ is said to be (Pareto) dominated by another vector ${\threshold}\in\RR^m$ if ${\threshold} \geq \threshold^*$ (component-wise) and there exists $i \in\titf{1}{m}$ such that $\threshold_{i} > {\threshold}^*_{i}$. Additionally, $\threshold^*$ is said to be strongly (Pareto) dominated by ${\threshold}$ if $\threshold > {\threshold}^*$ (component-wise). Therefore, the weak and strong Pareto fronts of any set $S \subset \RR^m$ are defined as follows:
\begin{itemize}

\item[$\bullet$] The strong Pareto front of $S$ is the set of all $\threshold^* \in S$, which are not dominated by another element of $S$:
$$\forall  \threshold\in S,\quad  \threshold\geq \threshold^*\quad\Longrightarrow\quad \threshold= \threshold^* .$$

\item[$\bullet$] The weak Pareto front of $S$ is the collection of all $\threshold^* \in S$, which are not strongly dominated by another element of $S$:
$$\forall \threshold\in S,\quad \exists i\in\titf{1}{m},\quad \threshold^*_i\geq \threshold_i.$$
\end{itemize}
The elements of the strong and weak Pareto front are called (respectively) strong and weak Pareto maxima.
 
Our goal in this paper is to study the weak and strong Pareto fronts of the set of robust sustainable thresholds $\TT(\ic)$ for a given initial condition $\ic\in\X$ \new{in order to obtain a full description of this set}. The approach we have taken is based on optimal control theory similarly as done in \cite{GajHer19}. The details are explained in the next section.

\begin{remark}\label{rem:charact}
\new{First, notice that strong Pareto maxima are also weak Pareto maxima. So, from \eqref{caractS} one can deduce that
\begin{equation}\label{eq:inclusion_fronts}
	 \pareto^\mathcal{S}\left(\TT(\ic)\right)+ \RR^m_- \subseteq \pareto^\mathcal{W}\left(\TT(\ic)\right)+ \RR^m_- \subseteq \TT(\ic),
\end{equation}
where $\pareto^\mathcal{S}\left(S\right)$ and $\pareto^\mathcal{W}\left(S\right)$ denote the strong and weak Pareto fronts of a set $S$, respectively. Thanks to the structure of the constraints \eqref{eq:constraints}-\eqref{eq:end_point} and the standing assumptions, we will be able to prove (see Theorem \ref{thm:pareto_threshols}) that for any $c\in\TT(\ic)$ one can find a strong Pareto maximum $c^*\in\TT(\ic)$ such that $c^*\geq c$. This in turn shows that the inclusions in \eqref{eq:inclusion_fronts} are attained as equalities, and therefore, the strong (or weak) Pareto front of $\TT(\ic)$ allows to recover the whole set of robust sustainable thresholds. Moreover, by the same arguments it follows that $\pareto^\mathcal{S}\left(\TT(\ic)\right)$ is the smallest subset of $\TT(\ic)$ that allows to recover $\TT(\ic)$ by adding the cone $\RR^m_-$, which means that the strong Pareto front of $\TT(\ic)$ can be interpreted in some sense as the extreme points of $\TT(\ic)$ whenever this set is convex.}
\end{remark}

 \subsection{Standing assumptions}

In this work, we assume that the data of the dynamical system \eqref{eq:system} with constraints \eqref{eq:constraints}-\eqref{eq:end_point} satisfy the following basic conditions, which we term \emph{standing assumptions}:
\begin{enumerate}[label=\textbf{(H\arabic*)}]
\item\label{hyp:1} $\dynamics_k(\cdot,\cdot,\scenario)$ is continuous on $\X\times\U$ for any $k\in \titf{0}{N}$ and any $\scenario\in\Omega_k$.
\item\label{hyp:2} For each $k\in \titf{0}{N}$ and $i\in\titf{0}{m}$:
\begin{itemize}
\item $\constraints^k_i$ is upper semicontinuous and bounded below on $\X\times\U$,
\item  $\fcost_i$ is upper semicontinuous and bounded below on $\X$.
\end{itemize}

\item\label{hyp:3} $\X$ is a finite-dimensional Banach space.
\item\label{hyp:4} $\U$ is a nonempty compact metric space.
\end{enumerate}

These hypotheses ensure that, for a given initial condition $\ic\in\X$, a scenario $\w\in\WW$ and a threshold vector $c\in\RR^m$, the set of feasible solutions to the dynamical system \eqref{eq:system}-\eqref{eq:constraints}-\eqref{eq:end_point} is compact in $\XX$. Since 
$\u\mapsto \x_{\ic}^\w(\u)$ is a continuous map, and $\UU$ is a compact metric space,

$$A^\w_\ic\defegal \left\{(\x,\u)\in\XX\times\UU\mid\x=\x_{\ic}^\w(\u)\right\}$$	
is a compact subset of $\XX\times\UU$. Furthermore, the set

$$B^\w_\threshold\defegal\left\{(\x,\u)\in\XX\times\UU~\middle|~\ \constraints^k(\state_k,\control_k)\geq \threshold, \ \forall k\in\titf{0}{N},\ \fcost(\state_{N+1})\geq c \right\}$$
is closed in $\XX\times\UU$. Now, since the set of admissible trajectories is exactly the projection of \pg{$A_\ic\cap B_\threshold$} over $\X$, we conclude that set of feasible solutions to the dynamical system \eqref{eq:system} with  constraints \eqref{eq:constraints}-\eqref{eq:end_point} is compact (possibly empty) in $\XX$. A similar argument shows that the set of admissible controls for the dynamical system \eqref{eq:system} with  constraints \eqref{eq:constraints}-\eqref{eq:end_point} is compact in $\UU$.

\section{the set of robust sustainable thresholds through optimal control}\label{sec:SRST_OC}

Let us consider a generic maximin optimal control problem and its corresponding optimal value (viewed as a function of the threshold vector):
 
 \begin{equation}\label{eq:vf}
\fv{\ic}{}(\threshold)\defegal\sup_{\u\in\UU}\left\{\inf_{\w\in\WW}\gcost{}(\x_{\ic}^\w(\u),\u) ~ \middle|\ \u\text{ and } \x_{\ic}^\w(\u)\text{ satisfy \eqref{eq:constraints}-\eqref{eq:end_point}, }\forall \w\in\WW\right\},
 \end{equation} 
where the function $\gcost{}: \XX\times\UU \to \RR$ is a generic function that will take a closed form when appropriate.

As it turns out, the set of robust sustainable thresholds $\TT(\ic)$ is the collection of all  thresholds $\threshold$ for which $\fv{\ic}{}(\threshold)$ is finite; this holds true for any bounded below and upper semicontinuous choice we make for the cost $\gcost{}$, as it is shown below.
 
  \begin{proposition}\label{prop:equivalence}Assume that $\gcost{}: \XX\times\UU \to \RR$ is bounded below and upper semicontinuous. Then, for any $\ic \in \X$ and $c\in\RR^m $, one has
 
	$$\threshold \in \TT(\ic) \quad\Longleftrightarrow\quad \fv{\ic}{}(\threshold)\in\RR.$$
	Furthermore, in any of these two cases, there is an optimal control for the optimization problem associated with  $\fv{\ic}{}(\threshold)$.
%	
%	$$\pi(\u):=\inf_{\w\in\WW}\left\{\gcost{}(\x_{\ic}^\w(\u),\u) ~ \middle|\ \u\text{ and } \x_{\ic}^\w(\u)\text{ satisfy \eqref{eq:constraints}-\eqref{eq:end_point}}\right\}.$$ 
 \end{proposition}
\begin{proof}
Let us point out that the mapping $\u\mapsto\x_{\ic}^\w(\u)$ is continuous for any $\ic\in\X$ and $\w\in\WW$ fixed, and so $\u\mapsto\gcost{}(\x_{\ic}^\w(\u),\u)$ is upper semicontinuous. This implies that the functional $\u\mapsto\displaystyle\inf_{\w\in\WW}\gcost{}(\x_{\ic}^\w(\u),\u)$ is upper semicontinuous too. Notice that this functional is also finite, because $\gcost{}$ is bounded below and $\WW$ is nonempty. Moreover, for any $\ic\in\X$ and $\w\in\WW$ fixed, let
$$\UU_\ic^\w(\threshold):=\left\{\u\in\UU\mid\u\text{ and } \x_{\ic}^\w(\u)\text{ satisfy \eqref{eq:constraints}-\eqref{eq:end_point}}\right\}$$
be the set of admissible controls for the dynamical system \eqref{eq:system} with  constraints \eqref{eq:constraints}-\eqref{eq:end_point}. As remarked earlier, this set is closed in $\UU$ for any given $\threshold\in\RR^m$, and therefore compact. In particular, since
$$\left\{\u\in\UU\mid\u\text{ and } \x_{\ic}^\w(\u)\text{ satisfy \eqref{eq:constraints}-\eqref{eq:end_point}, }\forall \w\in\WW\right\}=\bigcap_{\w\in\WW}\UU_\ic^\w(\threshold),$$
the set on the righthand side is compact as well.

Notice that $\threshold \in\TT(\ic)$ if and only if $\bigcap_{\w\in\WW}\UU_\ic^\w(\threshold)$ is nonempty. Therefore, the maximum in the definition of $\fv{\ic}{}(\threshold)$ is attained, and so $\fv{\ic}{}(\threshold)<+\infty$; the latter is a consequence of maximizing a finite and upper semicontinuous map over a nonempty compact set. Furthermore, if $\fv{\ic}{}(\threshold)<+\infty$, then clearly $\bigcap_{\w\in\WW}\UU_\ic^\w(\threshold)$ and the conclusion follows.

\end{proof}

This relation implies that if one wants to determine the set of robust sustainable thresholds, one may instead solve an optimization problem to check whether a given threshold is sustainable for a given initial state. Furthermore, this also suggests that, to compute the strong Pareto front of $\TT(\ic)$, one might try to construct a suitable functional $\gcost{}$ from the constraint mapping $\constraints$. Inspired by this idea, we provide a scheme for constructing the strong Pareto front of $\TT(\ic)$. The main feature of this procedure is that it works from the inside in the sense that, starting from a given sustainable threshold, it provides a strong Pareto maximum.

\subsection{A characterization of the strong Pareto front}

We now show that for a given initial state $\ic\in\X$, the strong Pareto maxima of $\TT(\ic)$ can be computed by solving a sequence of $m$ (the dimension of the constraint space) optimal control problems. For this purpose, we will construct a scheme by considering a sequence of maximin problems obtained by setting the functional $\gcost{}$ as
\begin{align}\label{eq:cost_pareto}
\gcost{}^i(\x,\u)\defegal\min\left\{\bigwedge_{k=0}^{N}\constraints^k_i(\state_k,\control_k),\fcost_i(\state_{N+1})\right\},\qquad\text{for some } i\in\titf{1}{m}.
\end{align}
where, for any $p,q\in\NN$ and $(a_k)_{k=p}^q$ we use the notation
$$\bigwedge_{k=p}^{q}a_k:=\min_{k=p,\ldots,q}a_k.$$
 
For the sake of exposition, let us introduce the mapping $\bfthreshold:\UU\to\RR^m$  given by
\begin{align}\label{eq:threshold_operator}
	\bfthreshold(\u)\defegal\left(\inf_{\w\in\WW}\gcost{}^1(\x_{\ic}^\w(\u),\u),\ldots,\inf_{\w\in\WW}\gcost{}^m(\x_{\ic}^\w(\u),\u)\right),\qquad\forall \u\in\UU.
\end{align}
\begin{remark}\label{rem:threshold_map}
Since we are assuming that each $\constraints_i$ is bounded below, it follows that the images of the mapping $\bfthreshold:\UU\to\RR^m$ introduced above are well-defined (the infima are finite) and they are actually sustainable thresholds. Indeed, let $\u=(\control_k)_{k=0}^{N}\in\UU$ be a control, $\w\in\WW$ a scenario and $\x_{\ic}^\w(\u)=(\state_k)_{k=0}^{N+1}$ be the corresponding trajectory. Then, in particular, $\bfthreshold(\u)\in\TT(\ic)$ because, by definition, we have
	$$\constraints^k(\state_k,\control_k)\geq \bfthreshold(\u)\quad \text{(component-wise)}, \ \forall k\in\titf{0}{N}$$
	and 
		$$\fcost(\state_{N+1})\geq \bfthreshold(\u)\quad \text{(component-wise)}.$$
\end{remark}

\begin{definition}\label{def:operator}
	Given $\ic\in\X$, a set-valued map $\pareto_\ic:\titf{1}{m}\times\TT(\ic)\rightrightarrows\RR^m$ is said to be a Pareto operator provided that
$$\pareto_\ic(i,c)=\left\{\bfthreshold(\u) ~|~\u\in\UU \mbox{ is optimal for }~\fv{\ic}{i}(\threshold)\right\}, $$
where %$\u^i\in\UU$ is an optimal control for
\begin{align}\label{eq:value_pareto}
	\fv{\ic}{i}(\threshold)\defegal\displaystyle\sup_{\u\in\UU}\left\{\inf_{\w\in\WW}\gcost{}^i(\x_{\ic}^\w(\u),\u)  ~ \middle|\ \u\text{ and } \x_{\ic}^\w(\u)\text{ satisfy \eqref{eq:constraints}-\eqref{eq:end_point}, }\forall \w\in\WW\right\}.
\end{align}
\end{definition}	
Note that in Definition \ref{def:operator}, $\gcost{}^i$ and $\bfthreshold$ are the mappings given by \eqref{eq:cost_pareto} and \eqref{eq:threshold_operator}, respectively. Let us now see that a Pareto operator has nonempty values, and moreover, its images are robust sustainable thresholds.

\begin{proposition}\label{prop:pareto_operator}
	For any $\ic\in\X$, a Pareto operator $\pareto_\ic:\titf{1}{m}\times\TT(\ic)\rightrightarrows\RR^m$ has nonempty values and
	$$\pareto_\ic(i,\threshold)\subset\TT(\ic),\qquad\forall i\in\titf{1}{m},\ \threshold\in\TT(\ic).$$
\end{proposition} 
\begin{proof}
	Since each functional $\gcost{}^i$ is bounded below and upper semicontinuous, by Proposition \ref{prop:equivalence} we have that for a given $\threshold\in\TT(\ic)$, there is an optimal control for the optimization problem associated with $\fv{\ic}{i}(\threshold)\in\RR$. Thus,  $\pareto_\ic(i,\threshold)$ is nonempty.
		
	Finally, the fact that $\pareto_\ic(i,\threshold) \subset \TT(\ic)$ follows by Remark \ref{rem:threshold_map}.
\end{proof}

With the notion of the Pareto operator at hand, we are now ready to introduce a scheme for finding strong Pareto maxima of the set of robust sustainable thresholds as claimed above. \new{In particular, the following theorem implies that for any $c\in\TT(\ic)$ one can find a strong Pareto maximum $c^*\in\TT(\ic)$ such that $c^*\geq c$ (component-wise) as claimed in Remark \ref{rem:charact}.}

\begin{theorem}\label{thm:pareto_threshols} For any initial condition $\ic\in\X$, sustainable threshold $\threshold^0\in\TT(\ic)$ and permutation $\sigma:\titf{1}{m}\to\titf{1}{m}$, consider the sequence $\threshold^1,\ldots,\threshold^m$ generated inductively by a Pareto operator $\pareto_\ic:\titf{1}{m}\times\TT(\ic)\rightrightarrows\TT(\ic)$ as follows:
$$\threshold^{i} \in \pareto_\ic(\sigma(i),\threshold^{i-1}),\quad i\in\titf{1}{m}.$$

Then, $\threshold^m$ belongs to the strong Pareto front of $\TT(\ic)$ with
$$\threshold^m\geq \threshold^{m-1}\geq\ldots\geq \threshold^{1}\geq \threshold^0\quad\text{(component-wise)}$$
and
$$\fv{\ic}{\sigma(i)}(\threshold^{i-1})= \threshold_{\sigma(i)}^{j},\qquad\forall i\in\titf{1}{m},\ j\in\titf{i}{m}.$$
Here, $\fv{\ic}{i}(\cdot)$ is given by \eqref{eq:value_pareto}. In particular, 
$$\threshold^m=\left(\fv{\ic}{1}\left(\threshold^{\sigma(1)-1}\right),\ldots,\fv{\ic}{m}\left(\threshold^{\sigma(m)-1}\right)\right).$$
\end{theorem}
\begin{proof}
For the sake of simplicity, let us consider only the case where the permutation $\sigma$ is the identity; that is, $\sigma(i)=i$ for any $i\in\titf{1}{m}$. For more general permutations, the proof is analogous (it is sufficient to redefine the constraint mappings $\constraints^0,\ldots,\constraints^N$ and $\fcost$ by changing the order of their components).

Note that the sequence $\threshold^1,\ldots,\threshold^m$ is well-defined. Indeed, this is a straightforward consequence of the induction principle and Proposition \ref{prop:pareto_operator}. In particular, we have $\threshold^1,\ldots,\threshold^m\in\TT(\ic)$.

Let us continue by showing that $\threshold^{i}\geq \threshold^{i-1}$ for any $i\in\titf{1}{m}$. Given that a control $\u^i$ in the definition of $\pareto_\ic(i,\threshold^{i-1})$ (see Definition \ref{def:operator}) is, in particular, a feasible control for problem $\fv{\ic}{i}(\threshold^{i-1})$; we have
 $$\threshold^i_j=\inf_{\w\in\WW}\min\left\{\bigwedge_{k=0}^{N}\constraints^k_j(\state^i_k,\control^i_k),\fcost_j(\state^i_N)\right\}\geq\threshold^{i-1}_j,\qquad\forall j\in\titf{1}{m},
$$
where $\u^i=(\control_k^i)_{k=0}^{N}$ and $\x_{\ic}^\w(\u^i)=(\state_k^i)_{k=0}^{N+1}$. Therefore, we have $\threshold^{i}\geq \threshold^{i-1}$ (component-wise) for any $i\in\titf{1}{m}$.
 	
Let us now prove that $\threshold^m$ is a strong Pareto maximum. Let $c\in\TT(\ic)$ be such that $\threshold\geq \threshold^m$. Assume for  the sake of contradiction that $\threshold\neq \threshold^m$. Let $i\in\{1,\ldots,m\}$ be an index such that $\threshold_i>\threshold^m_i$. In particular, we have
$$\threshold_i>\threshold^m_i\geq \threshold^i_i=\fv{\ic}{i}(\threshold^{i-1}).$$
However, since $c\in\TT(\ic)$, there is a control $\u=(\control_k)_{k=0}^{N}\in\UU$ such that $\u$ and $\x_{\ic}^\w(\u)=(\state_k)_{k=0}^{N+1}$ satisfy for any scenario $\w\in\WW$ that
$$\min\left\{\bigwedge_{k=0}^{N}\constraints^k(\state_k,\control_k),\fcost(\state_{N+1})\right\}\geq \threshold\geq \threshold^m\geq c^{i-1}.$$
This means that $\u$ is feasible for the optimal control problem associated with $\fv{\ic}{i}(\threshold^{i-1})$, and thus by definition 
$$\fv{\ic}{i}(\threshold^{i-1})\geq \inf_{\w\in\WW}\min\left\{\bigwedge_{k=0}^{N}\constraints^k_i(\state_k,\control_k),\fcost_i(\state_{N+1})\right\}\geq c_i;$$ this leads to a contradiction. Therefore, $\threshold=\threshold^m$, and consequently, $\threshold^m$ is a strong Pareto maximum.
 	
Finally, we have by definition that $\threshold^i_i=\fv{\ic}{i}(\threshold^{i-1})$ for any $i\in\titf{1}{m}$. Thus, let $i\in\titf{1}{m-1}$ and $j\in\titf{i+1}{m}$. Then, since $\threshold^{j}\geq\threshold^{i}\geq\threshold^{i-1}$, we have that $\u^{j}$, the optimal control given in Definition \ref{def:operator}, and its corresponding optimal trajectory are feasible for the optimal control problem associated with $\fv{\ic}{i}(\threshold^{i-1})$. In particular, we must have
\begin{equation*}
\begin{aligned}
	\fv{\ic}{i}(\threshold^{i-1})&=\inf_{\w\in\WW}\min\left\{\bigwedge_{k=0}^{N}\constraints^k_i(\state^i_k,\control^i_k),\fcost_i(\state^i_N)\right\}\\
	&\geq \inf_{\w\in\WW}\min\left\{\bigwedge_{k=0}^{N}\constraints^k_i(\state^j_k,\control^j_k),\fcost_i(\state^j_N)\right\}\\&=\threshold^j_i\geq \threshold^i_i.	
\end{aligned}
\end{equation*}

Since, $\threshold^i_i=\fv{\ic}{i}(\threshold^{i-1})$, the conclusion follows
\end{proof}

To give an idea of what the sequence $\threshold^1,\ldots,\threshold^m$ generated by the preceding theorem looks like, we describe a situation with a threshold space of dimension $m=2$ in Figure \ref{fig:sP}.

\begin{figure}[t]
\begin{center}
\psset{xunit=.5cm, yunit=.5cm, linewidth=.2pt}

\begin{pspicture}(1,-2)(12,13)
\psaxes[linecolor=black,  arrows=->, linewidth=1.5pt,Ox=,Oy=,Dx=20,Dy=2,dx=2.5,dy=3,labels=none,tickstyle=bottom,ticksize=.2](-1,0)(15,12)
%%TICK LABELS

%%CURVE  (LARGE)
\psline[linewidth=1pt, linecolor=lightgreen,fillstyle=solid,fillcolor=lightgreen](-0.95,11)(3,11)(7,8)(7,0.1)(-0.95,0.1)
\psline[linewidth=1pt, linecolor=lightgreen,fillstyle=solid,fillcolor=lightgreen](7,3)(12.8,0.1)(7,0.1)
\pscurve[linewidth=1pt,fillstyle=solid,fillcolor=lightgreen](3,11)(6,10)(7,8)
\pscurve[linewidth=1pt,fillstyle=solid,fillcolor=lightgreen](7,3)(10,2)(13,0)  
  \psline[linewidth=1.5pt, linecolor=blue,linestyle=dashed,arrows=->](1,5)(7.1,5)
    \psline[linewidth=1.5pt, linecolor=blue,linestyle=dashed,arrows=->](7,5)(7,8)
    
      \psline[linewidth=1.5pt, linecolor=red,linestyle=dashed,arrows=->](1,5)(1,11)
    \psline[linewidth=1.5pt, linecolor=red,linestyle=dashed,arrows=->](1,11)(3,11)
\rput(4,2) {\textcolor{darkgreen}{\large $\TT(\ic)$}}
\rput[tr](0.8,5.2) {\textcolor{black}{$\threshold^0$}}
\rput[tl](7.2,5) {\textcolor{black}{ $\pareto_\ic(1,\threshold^0)$}}
\rput[bl](7.2,8) {\textcolor{black}{ $\pareto_\ic(2,\threshold^1)$}}
\rput[bl](-0.5,11.2) {\textcolor{black}{ $\pareto_\ic(2,\threshold^0)$}}
\rput[bl](3,11.2) {\textcolor{black}{ $\pareto_\ic(1,\threshold^1)$}}
 % \psline[fillstyle=solid,fillcolor=gray,linewidth=1.5pt, linecolor=gray,linestyle=dashed, opacity=0.5](10,5)(5,5)(5,2)(10,2)(10,5)
\pscircle[linewidth=1pt,fillstyle=solid,fillcolor=lightgreen](7,5){2pt}
  \psdot[dotsize=4pt](7,8)
    \psdot[dotsize=4pt](3,11)
    \psdot[dotsize=4pt](1,5)
      \pscircle[linewidth=1pt,fillstyle=solid,fillcolor=lightgreen](7,3){2pt}
\pscircle[linewidth=1pt,fillstyle=solid,fillcolor=lightgreen](1,11){2pt}

% LABELS - TEXT
\rput(-1.9,12.1) {$\threshold_2$}
\rput(16,0) {$\threshold_1$}

\end{pspicture}
 \end{center}
 \caption{Sketch of sequence of thresholds generated by Theorem \ref{thm:pareto_threshols} for a problem with $m=2$. The black line indicates the strong Pareto front of the example. Here, $\pareto_\ic(1,\threshold^1)$ and $\pareto_\ic(2,\threshold^1)$ are the two strong Pareto maxima found by the scheme using the two possible permutations on $\titf{1}{2}$ starting from $\threshold^0\in\TT(\ic)$.}\label{fig:sP}
\end{figure}
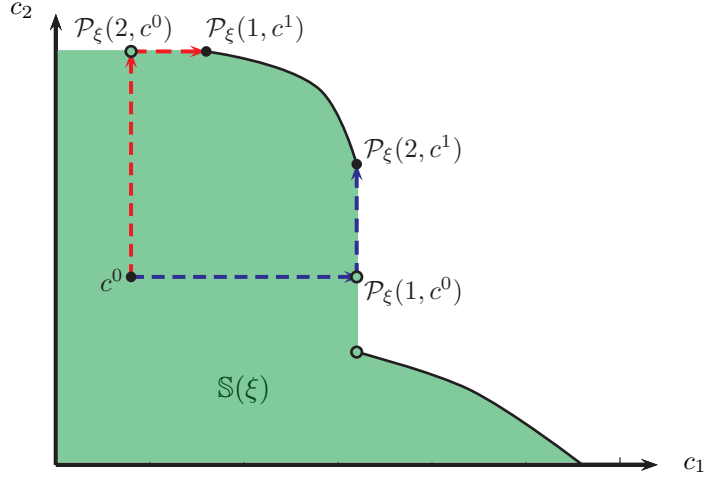

\begin{remark}
\new{In Theorem \ref{thm:pareto_threshols}, it is not difficult to see that if  for an initial condition $\ic\in\X$, the  thresholds vector $\threshold^0\in\TT(\ic)$ already belongs to the strong Pareto front of  $\TT(\ic)$, then by definition of strong Pareto maxima, the sequence of thresholds $\threshold^1,\ldots,\threshold^m$ generated by the proposed method is equal to $\threshold^0$.}
\end{remark}

\section{The weak Pareto front}\label{weakpf}

Let us now focus on the weak Pareto front of the set of robust sustainable thresholds. As with the strong Pareto front, we will present a method for computing the weak Pareto front by means of optimal control tools. However, in this case, we will consider a method that identifies elements in the set from the outside; that is, we will construct a weak Pareto maximum from threshold vectors that are not sustainable for the given initial condition.
 In particular, this means that the optimal control problems we are considering do not require forcing the constraints and thus are unconstrained problems. This helps somewhat in reducing the computational time.

To begin, we introduce the optimal control problem 
\begin{equation}\label{eq:Wc}
\aux{\ic}{\threshold}=\max_{\u\in\UU}\inf_{\w\in\WW}\left\{ \min\left\{\bigwedge_{k=0}^{N}\Phi_k^\threshold(\state_k,\control_k),\Theta^\threshold(\state_{N+1})\right\}~\middle|~\ \x_{\ic}^\w(\u)=(\state_k)_{k=0}^{N+1}\right\}
\end{equation}
where $\threshold\in\RR^m$ is a given threshold vector, $\Phi_k^\threshold:\X\times\U\to\RR$ is given for any $k\in\titf{0}{N}$ by
$$ \Phi_k^\threshold(\state,\control)=\bigwedge_{i=1}^m\constraints^k_i(\state,\control)-c_i$$
and
$\Theta^\threshold:\X\to\RR$ is given by
$$ \Theta^\threshold(\state)=\bigwedge_{i=1}^m\fcost_i(\state)-c_i$$

Note that, since $\UU$ is (nonempty) compact and the functional to be maximized in the definition of $\aux{\ic}{\threshold}$ is upper semicontinuous (it is the infimum of upper semicontinuous functions that depend on each scenario $\w$), the use of the maximum instead of the supremum in $\aux{\ic}{\threshold}$ is justified, which means that the optimal value $\aux{\ic}{\threshold}$ is attained at some optimal control $\u\in\UU$. Moreover, this implies that for any $\ic\in\X$, we have 
\begin{equation}\label{eq:equiv}
\threshold \in \TT(\ic)~\Longleftrightarrow~ \aux{\ic}{\threshold} \geq 0.
\end{equation}
This equivalence shows the strong link between the level-set of the value function $\aux{\ic}{\threshold}$ and the set $\TT(\ic)$. Furthermore, there is a rather straightforward way to construct a point in the weak Pareto front of $\TT(\ic)$ from any given \emph{unsustainable} threshold $\threshold \in\RR^m$ through the value function $\aux{\ic}{\threshold}$. In this context, unsustainable means that $\aux{\ic}{\threshold}<0$, which is equivalent to saying that for any control $\u\in\UU$ there is a scenario $\w\in\WW$, an index $i\in\titf{1}{m}$ and an instant $k\in\titf{0}{N}$ such that $\constraints^k_i(\state_k,\control_k)<\threshold_i$ or $\fcost_i(\state_{N+1})<\threshold_i$, where $\u=(\control_k)_{k=0}^{N}$  and $\x_{\ic}^\w(\u)=(\state_k)_{k=0}^{N+1}$. We describe this situation in Figure \ref{fig:wP}.

\begin{figure}[t]
\begin{center}
\psset{xunit=.5cm, yunit=.5cm, linewidth=.2pt}

\begin{pspicture}(1,-2)(12,13)
\psaxes[linecolor=black,  arrows=->, linewidth=1.5pt,Ox=,Oy=,Dx=20,Dy=2,dx=2.5,dy=3,labels=none,tickstyle=bottom,ticksize=.2](0,0)(15,12)
%
%%TICK LABELS
%
%%CURVE  (LARGE)
\psline[linewidth=1pt, linecolor=lightgreen,fillstyle=solid,fillcolor=lightgreen](0.05,9)(1,9)(7,6)(7,0.1)(0.05,0.1)
\psline[linewidth=1pt, linecolor=lightgreen,fillstyle=solid,fillcolor=lightgreen](7,3)(7,0.1)(13,0.1)
 \psline[linewidth=1.5pt, linestyle=dashed](0,9)(1,9)
 \pscurve[linewidth=1.5pt,fillstyle=solid,fillcolor=lightgreen,linestyle=dashed](1,9)(5,8)(7,6)
 \psline[linewidth=1.5pt, linestyle=dashed](7,6)(7,3)
 \pscurve[linewidth=1.5pt,fillstyle=solid,fillcolor=lightgreen,linestyle=dashed](7,3)(12,1)(13,0)

\rput(3,2) {\textcolor{darkgreen}{\large $\TT(\ic)$}}
\rput[bl](8.2,11.2) {\textcolor{black}{$\threshold^0$}}
\rput[tr](5.2,7.8) {\textcolor{black}{ $p\left(\threshold^0\right)$}}
\rput[tr](6.6,6.2) {\textcolor{black}{ $p\left(\threshold^1\right)$}}
\rput[bl](10.2,10.2) {\textcolor{black}{ $\threshold^1$}}
\rput[tr](7,3) {\textcolor{black}{ $p\left(\threshold^2\right)$}}
\rput[l](12.2,8.2) {\textcolor{black}{ $\threshold^2$}}
\pscircle[linewidth=1pt,fillstyle=solid,fillcolor=lightred](5,8){2pt}
    \psdot[dotsize=4pt](8,11)
    \psline[linewidth=1pt, linestyle=dashed,arrows=>,linecolor=red](5.2,8.2)(7.8,10.8)
\pscircle[linewidth=1pt,fillstyle=solid,fillcolor=lightred](6.55,6.55){2pt}
    \psdot[dotsize=4pt](10,10)
    \psline[linewidth=1pt, linestyle=dashed,arrows=>,linecolor=red](6.6,6.6)(9.8,9.8)
 \pscircle[linewidth=1pt,fillstyle=solid,fillcolor=lightred](7,3){2pt}
    \psdot[dotsize=4pt](12,8)
    \psline[linewidth=1pt, linestyle=dashed,arrows=>,linecolor=red](7.2,3.2)(11.8,7.8)
%
%
%% LABELS - TEXT
\rput(-0.9,12.1) {$\threshold_2$}
\rput(16,-0.7) {$\threshold_1$}
\end{pspicture}
 \end{center}
\caption{Sketch of Theorem \ref{thm:weak_front} for a problem with $m=2$ on the construction of weak Pareto maxima from unsustainable thresholds $\threshold^0$, $\threshold^1$ and $\threshold^2$. The dashed black line indicates the weak Pareto front of the example. }\label{fig:wP}
\end{figure}
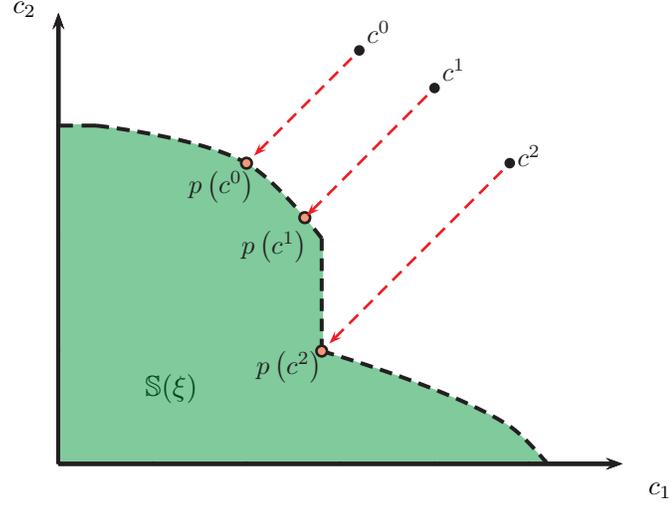

\begin{theorem}\label{thm:weak_front}
Let $\threshold^*,\bar \threshold \in\RR^m$. We then have the following:
\begin{enumerate}
\item\label{thm:weak_front.1}  $\threshold^*$ is a weak Pareto maximum of $\TT(\ic)$ if and only if $\aux{\ic}{\threshold^*} = 0$. 
\item\label{thm:weak_front.3} If $\aux{\ic}{\bar \threshold} < 0$, then $$p(\bar\threshold)\defegal \bar\threshold + \aux{\ic}{\bar\threshold} \bfone$$ 
belongs to the weak Pareto front of $\TT(\ic)$, where $\bfone=(1,1,\ldots,1) \in \RR^m$.
\end{enumerate}
\end{theorem}
\begin{proof}
Let us set for any $\threshold\in\RR^m$
\begin{equation}\label{exR}
 \R{}{\threshold}(\x,\u)\defegal \bigwedge_{i=1}^m\left[\min\left\{\bigwedge_{k=0}^{N}\constraints^k_i(\state_k,\control_k),\fcost_i(\state_{N+1})\right\}-\threshold_i\right],\quad\forall \x\in\XX,\ \u\in\UU.
  \end{equation} 
  In particular, it follows that
  $$\aux{\ic}{\threshold} =  \sup_{\u\in\UU}\inf_{\w\in\WW} \R{}{\threshold}(\x_{\ic}^\w(\u),\u).$$
  
\begin{enumerate}
\item Let us assume first that $\aux{\ic}{\threshold^*}=0$; then by \eqref{eq:equiv}, we obtain $\threshold^* \in \TT(\ic)$. To see that $\threshold^*$ is a weak Pareto maximum; suppose for  the sake of contradiction that there exists $\threshold \in \TT(\ic)$ with $ \threshold > \threshold^*$. We define
$$\delta \defegal \bigwedge_{i=1}^{m}\{\threshold_i-  \threshold^*_i\} > 0.$$
Since $\threshold \in \TT(\ic)$, there exists $\u=(\control_k)_{k=0}^{N}\in \UU$ such that for any scenario $\w\in\WW$ we have
$$\constraints^k_i(\state_k,\control_k)\geq \threshold_i,\qquad\forall i\in\titf{1}{m},\ k\in\titf{0}{N},$$
and 
$$\fcost_i(\state_{N+1})\geq \threshold_i,\qquad\forall i\in\titf{1}{m}$$
where $\x_{\ic}^\w(\u)=(\state_k)_{k=0}^{N+1}$. This implies that
$$0=\aux{\ic}{\threshold^*}\geq \inf_{\w\in\WW} \R{}{ \threshold^*}(\x_{\ic}^\w(\u),\u)\geq\bigwedge_{i=1}^m \threshold_i-\threshold^*_i =  \delta  ,$$
which is a contradiction. Therefore, $\threshold^*$ is a weak Pareto maximum of $\TT(\ic)$.

On the other hand, assume now that $\threshold^*$ is a weak Pareto maximum of $\TT(\ic)$. In particular, $\threshold^*\in\TT(\ic)$, so \eqref{eq:equiv} yields $\aux{\ic}{\threshold^*}\geq 0$. Let $\u\in\UU$ be an optimal control for $ \aux{\ic}{\threshold^*}$; that is, $\displaystyle\inf_{\w\in\WW}\R{}{\threshold^*}(\x_{\ic}^\w(\u),\u)=\aux{\ic}{\threshold^*}$. Suppose by contradiction that $\displaystyle\inf_{\w\in\WW}\R{}{\threshold^*}(\x_{\ic}^\w(\u),\u)>0$ and define 
$$\threshold\defegal\threshold^*+\frac{1}{2}\displaystyle\inf_{\w\in\WW}\R{}{\threshold^*}(\x_{\ic}^\w(\u),\u)\bfone> \threshold^*.$$
Notice that
$$\aux{\ic}{\threshold}  \geq \inf_{\w\in\WW}\R{}{\threshold}(\x_{\ic}^\w(\u),\u)=\inf_{\w\in\WW}\R{}{\threshold^*}(\x_{\ic}^\w(\u),\u)-\frac{1}{2}\inf_{\w\in\WW}\R{}{\threshold^*}(\x_{\ic}^\w(\u),\u).$$
Then, we can deduce that 
$$\aux{\ic}{\threshold}  \geq \frac{1}{2}\inf_{\w\in\WW}\R{}{\threshold^*}(\x_{\ic}^\w(\u),\u)\geq 0.$$
In particular, we must have that $\threshold\in\TT(\ic)$. Therefore, we have found some robust sustainable threshold $c$ for which $c>c^*$, which contradicts the fact that $c^*$ is assumed to be a weak Pareto maximum. 
\item It is straightforward to see that $\aux{\ic}{p(\bar \threshold)}=\aux{\ic}{\bar \threshold}-\aux{\ic}{\bar \threshold}=0$. Thus, in light of the first statement of Theorem \ref{thm:weak_front}, the conclusion follows.\end{enumerate}
\end{proof}

\begin{remark}
\pg{Notice that the first part of Theorem \ref{thm:weak_front} provides a characterization of the weak Pareto maxima of $\TT(\ic)$. In addition, one can see that 
if }$ \aux{\ic}{\threshold}>0$, then $\threshold$ belongs to the interior of $\TT(\ic)$. This comes directly from the proof of Theorem \ref{thm:weak_front} and the fact that $\threshold+\frac{1}{2}\aux{\ic}{\threshold}\bfone \in\TT(\ic)$ implies
$$\threshold+\frac{1}{2}\aux{\ic}{\threshold}[-1,1]^m \subset\TT(\ic).$$
\end{remark}

 \subsection{Dynamic programming principle}
 
 To compute the optimal value $\aux{\ic}{\threshold}$, we use the dynamic programming principle. This method leads to an implementable way to compute the optimal value $\aux{\ic}{\threshold}$ and, consequently, to a practical way to compute the set of robust sustainable thresholds and its weak Pareto front. 
 
For the sake of exposition, for any $n\in\titf{0}{N}$, we write
$$V^\threshold_n(\ic)\defegal\max_{\u\in\UU}\inf_{\w\in\WW}\left\{ \min\left\{\bigwedge_{k=n}^{N}\Phi_k^\threshold(\state_k,\control_k),\Theta^\threshold(\state_{N+1})\right\}~\middle|~\ \begin{matrix}\state_{k+1}=\dynamics_k(\state_k,\control_k,\scenario_k)\\ k\in\titf{n}{N},\ \state_{n}=\ic\end{matrix}%\x_{\ic}^\w(\u)=(\state_k)_{k=0}^{N+1}
\right\}$$	
and we set $V^\threshold_{N+1}(\ic)=\Theta^\threshold(\ic)$. Notice that, for any $\ic\in\X$, we have $V^\threshold_0(\ic)=\aux{\ic}{\threshold}
$.
 To obtain $\aux{\ic}{\threshold}$, we use the dynamic programming principle for computing $V^\threshold_0(\ic)$ from the sequence of value functions $V^\threshold_{1}(\cdot),\ldots,V^\threshold_{N+1}(\cdot)$. 
 
\begin{remark}\label{rem:usc_vf}
Notice that the mapping $\ic\mapsto V^\threshold_n(\ic)$ is upper semicontinuous for any $n\in\titf{0}{N+1}$. Indeed, the case $n=N+1$ is obvious since $\Theta^\threshold$ is clearly upper semicontinuous.  If $n\in\titf{0}{N}$, then for any sequence $\{\ic^j\}_{j\in\NN}$ that converges to $\bar\ic$, there is a sequence of optimal controls $\{(\control^j_k)_{k=0}^N\}_{j\in\NN}\subseteq\UU$ such that 
\begin{equation}\label{eq:usc_aux}
V^\threshold_n(\ic^j)=\inf_{\w\in\WW}\left\{ \min\left\{\bigwedge_{k=n}^{N}\Phi_k^\threshold(\state_k,\control^j_k),\Theta^\threshold(\state_{N+1})\right\}~\middle|~\ \begin{matrix}\state_{k+1}=\dynamics_k(\state_k,\control^j_k,\scenario_k)\\ k\in\titf{n}{N},\ \state_{n}=\ic^j\end{matrix}%\x_{\ic}^\w(\u)=(\state_k)_{k=0}^{N+1}
\right\}
\end{equation}
By compactness of $\UU$, we can assume that  $\{(\control^j_k)_{k=0}^N\}_{j\in\NN}$ converges to some  $(\bar\control_k)_{k=0}^N$. Notice as well that the mapping
$$(\ic,\u)\mapsto\inf_{\w\in\WW}\left\{ \min\left\{\bigwedge_{k=n}^{N}\Phi_k^\threshold(\state_k,\control_k),\Theta^\threshold(\state_{N+1})\right\}~\middle|~\ \begin{matrix}\state_{k+1}=\dynamics_k(\state_k,\control_k,\scenario_k)\\ k\in\titf{n}{N},\ \state_{n}=\ic\end{matrix}%\x_{\ic}^\w(\u)=(\state_k)_{k=0}^{N+1}
\right\}$$
is upper semicontinuous, and so, by taking limsup in \eqref{eq:usc_aux} we get
$$\limsup_{j\to+\infty}V^\threshold_N(\ic^j)\leq \inf_{\w\in\WW}\left\{ \min\left\{\bigwedge_{k=n}^{N}\Phi_k^\threshold(\state_k,\bar\control_k),\Theta^\threshold(\state_{N+1})\right\}~\middle|~\ \begin{matrix}\state_{k+1}=\dynamics_k(\state_k,\bar\control_k,\scenario_k)\\ k\in\titf{n}{N},\ \state_{n}=\bar\ic\end{matrix}%\x_{\ic}^\w(\u)=(\state_k)_{k=0}^{N+1}
\right\}.$$ Whence, using the definitions of the value funcion $V^\threshold_n(\bar\ic)$ we conclude.
\end{remark}
 
The dynamic programming principle for maximin problems such as the one that determines the value functions $(V^\threshold_n)_{n=0}^{N+1}$ is a well-known fact; see for example \cite{DLD}. We provide its proof for the sake of completeness.  In this setting, the dynamic programming principle reads as follows.
\begin{proposition}\label{prop:dpp} 
For any $n\in\titf{0}{N}$, $\threshold\in\RR^m$ and $\ic\in\X$, we have	
\begin{equation}\label{eq:ddp}
V^\threshold_n(\ic)=\max_{u\in\U}\min\left\{\inf_{\scenario\in\Omega_n}V^\threshold_{n+1}\left(\dynamics_n(\ic,\control,\scenario)\right),\Phi_n^\threshold(\ic,\control)\right\}.	
\end{equation}
	
\end{proposition}
\begin{proof}
For any  $n\in\titf{0}{N}$, let us defined recursively the function
$$W^\threshold_n(\ic)\defegal \sup_{u\in\U}\min\left\{\inf_{\scenario\in\Omega_n}V^\threshold_{n+1}\left(\dynamics_n(\ic,\control,\scenario)\right),\Phi_n^\threshold(\ic,\control)\right\}.$$ 	
The value  $W^\threshold_n(\ic)$ agrees with the right-hand side of \eqref{eq:ddp} with the supremum instead of the maximum. Let us check that this maximum is attained. Since $V^\threshold_{n+1}$ and $\Phi_n^\threshold(\ic,\cdot)$ are upper semicontinuous and $\dynamics_n(\ic,\cdot,\scenario)$ is continuous for any $n\in\titf{0}{N}$, we get that the functional to be maximized in the definition of $W^\threshold_n(\ic)$ is upper semicontinuous. Thus, since $\U$ is compact, this maximum is attained.

%Therefore, using the fact that $\Phi_k^\threshold(\ic,\cdot)$ and $V^\threshold_{k+1}$ is upper semicontinuous for every $k\in\titf{0}{N}$, it is not difficult to see by (backward) induction that for any $n\in\titf{0}{N}$, the mapping $$\control\mapsto\min\left\{\inf_{\scenario\in\Omega_n}V^\threshold_{n+1}\left(\dynamics_n(\ic,\control,\scenario)\right),\Phi_n^\threshold(\ic,\control)\right\}$$ is upper semicontinuous and consequently, by compactness of the control space $\U$, supremum in $W_n$ is always attained for any $n\in\titf{0}{N}$.

Let us prove by (backward) induction that $V^\threshold_n(\ic)=W^\threshold_n(\ic)$ for any $n\in\titf{0}{N}$.

In the sequel, unless otherwise stated, we use the notation  $\u=(\control_k)_{k=0}^{N}\in\UU$ and  $\w=(\scenario_k)_{k=0}^{N}\in\WW$. Let us first check that the dynamic programming principle holds for the case $n=N$. By definition we have
\begin{equation*}
\begin{aligned}
 V^\threshold_N(\ic)&\defegal\max_{\u\in\UU}\inf_{\w\in\WW}\left\{ \min\left\{\Phi_N^\threshold(\ic,\control_N),\Theta^\threshold(\dynamics_N(\ic,\control_N,\scenario_N))\right\}
\right\},\\
&=\max_{\control\in\U}\inf_{\scenario\in\Omega_N}\left\{ \min\left\{\Phi_N^\threshold(\ic,\control),\Theta^\threshold(\dynamics_N(\ic,\control,\scenario))\right\}\right\}\\
&=\max_{\control\in\U} \min\left\{\Phi_N^\threshold(\ic,\control),\inf_{\scenario\in\Omega_N}\Theta^\threshold(\dynamics_N(\ic,\control,\scenario))\right\}	
\end{aligned}
\end{equation*}
Since $V^\threshold_{N+1}\equiv\Theta^\threshold$, we get $V^\threshold_N(\ic)=W^\threshold_N(\ic)$. 

Let $n\in\titf{0}{N-1}$ and assume that $V^\threshold_k(\state)=W^\threshold_k(\state)$ for any $k\in\titf{n+1}{N}$ and any $\state\in\X$. Let us verify that $V^\threshold_n(\ic)=W^\threshold_n(\ic)$

%Similarly, as in \eqref{exR}, we set $$ \R{n}{\threshold}(\x,\u)\defegal \min\left\{\bigwedge_{k=n}^{N}\Phi_k^\threshold(\state_k,\control_k),\Theta^\threshold(\state_{N+1})\right\},\quad\forall \x\in\XX,\ \u\in\UU,$$ as the functional to be maximized in the definition of $V^\threshold_{n}$. It is then clear that $\R{n+1}{\threshold}(\x,\u)$ does not depend on $\control_n$ and also that $$ \R{n}{\threshold}(\x,\u) =\max\left\{\R{n+1}{\threshold}(\x,\u),\Phi_n^\threshold(\state_{n},\control_n)\right\}.$$

Let $\bar\u=(\bar\control_k)_{k=0}^{N}\in\UU$ be an optimal control for $ V^\threshold_n(\ic)$. Notice then that
\begin{eqnarray*}
 V^\threshold_n(\ic)=\inf_{\w\in\WW}\left\{ \min\left\{\Phi_n^\threshold(\ic,\bar\control_n),\bigwedge_{k=n+1}^{N}\Phi_k^\threshold(\state_k,\bar\control_k),\Theta^\threshold(\state_{N+1})\right\}\middle|\  \begin{matrix}\state_{k+1}=\dynamics_k(\state_k,\bar\control_k,\scenario_k)\\  k\in\titf{n+1}{N},\\\state_{n+1}=\dynamics_n(\ic,\bar\control_n,\scenario_n)\end{matrix}
\right\}\\
	=\min\left\{\Phi_n^\threshold(\ic,\bar\control_n),\inf_{\w\in\WW}\left\{\min\left\{ \bigwedge_{k=n+1}^{N}\Phi_k^\threshold(\state_k,\bar\control_k),\Theta^\threshold(\state_{N+1})\right\}\middle|\  \begin{matrix}\state_{k+1}=\dynamics_k(\state_k,\bar\control_k,\scenario_k)\\  k\in\titf{n+1}{N},\\\state_{n+1}=\dynamics_n(\ic,\bar\control_n,\scenario_n)\end{matrix}
\right\}\right\}
%\\\leq\min\left\{\Phi_n^\threshold(\ic,\bar\control_n),\max_{\u\in\UU}\inf_{\w\in\WW}\left\{ \min\left\{ \bigwedge_{k=n+1}^{N}\Phi_k^\threshold(\state_k,\control_k),\Theta^\threshold(\state_{N+1})\right\}\middle|\  \begin{matrix}\state_{k+1}=\dynamics_k(\state_k,\control_k,\scenario_k)\\  k\in\titf{n+1}{N}\\\state_{n+1}=\dynamics_n(\ic,\bar\control_n,\scenario_n)\end{matrix}
%\right\}\right\}
\end{eqnarray*}
%\begin{eqnarray*}
% V^\threshold_n(\ic)\defegal\max_{(\control,\u)\in\U\times\UU}\inf_{\w\in\WW}\left\{ \min\left\{\Phi_n^\threshold(\ic,\control),\bigwedge_{k=n+1}^{N}\Phi_k^\threshold(\state_k,\control_k),\Theta^\threshold(\state_{N+1})\right\}\middle|\  \begin{matrix}\state_{k+1}=\dynamics_k(\state_k,\control_k,\scenario_k)\\  k\in\titf{n+1}{N},\\\state_{n+1}=\dynamics_n(\ic,\control,\scenario_n)\end{matrix}
%\right\}\\
%	=\max_{(\control,\u)\in\U\times\UU}\min\left\{\Phi_n^\threshold(\ic,\control),\inf_{\w\in\WW}\left\{\min\left\{ \bigwedge_{k=n+1}^{N}\Phi_k^\threshold(\state_k,\control_k),\Theta^\threshold(\state_{N+1})\right\}\middle|\  \begin{matrix}\state_{k+1}=\dynamics_k(\state_k,\control_k,\scenario_k)\\  k\in\titf{n+1}{N},\\\state_{n+1}=\dynamics_n(\ic,\control,\scenario_n)\end{matrix}
%\right\}\right\}\\
%\leq\max_{\control\in\U}\min\left\{\Phi_n^\threshold(\ic,\control),\max_{\u\in\UU}\inf_{\w\in\WW}\left\{ \min\left\{ \bigwedge_{k=n+1}^{N}\Phi_k^\threshold(\state_k,\control_k),\Theta^\threshold(\state_{N+1})\right\}\middle|\  \begin{matrix}\state_{k+1}=\dynamics_k(\state_k,\control_k,\scenario_k)\\  k\in\titf{n+1}{N}\\\state_{n+1}=\dynamics_n(\ic,\control,\scenario_n)\end{matrix}
%\right\}\right\}
%\end{eqnarray*}
Furthermore, it also holds that
\begin{equation*}
\begin{aligned}
&\inf_{\w\in\WW}\left\{\min\left\{ \bigwedge_{k=n+1}^{N}\Phi_k^\threshold(\state_k,\bar\control_k),\Theta^\threshold(\state_{N+1})\right\}\middle|\  \begin{matrix}\state_{k+1}=\dynamics_k(\state_k,\bar\control_k,\scenario_k)\\  k\in\titf{n+1}{N},\\\state_{n+1}=\dynamics_n(\ic,\bar\control_n,\scenario_n)\end{matrix}
\right\}\\
%&\leq\max_{\u\in\UU}\inf_{\w\in\WW}\left\{ \min\left\{ \bigwedge_{k=n+1}^{N}\Phi_k^\threshold(\state_k,\control_k),\Theta^\threshold(\state_{N+1})\right\}\middle|\  \begin{matrix}\state_{k+1}=\dynamics_k(\state_k,\control_k,\scenario_k)\\  k\in\titf{n+1}{N},\\\state_{n+1}=\dynamics_n(\ic,\control,\scenario_n)\end{matrix}\right\}\\
&=\inf_{\scenario\in\Omega_n}\inf_{\w\in\WW}\left\{ \min\left\{ \bigwedge_{k=n+1}^{N}\Phi_k^\threshold(\state_k,\bar\control_k),\Theta^\threshold(\state_{N+1})\right\}\middle|\  \begin{matrix}\state_{k+1}=\dynamics_k(\state_k,\bar\control_k,\scenario_k)\\  k\in\titf{n+1}{N},\\\state_{n+1}=\dynamics_n(\ic,\bar\control_n,\scenario)\end{matrix}
\right\}\\
&\leq 
	\inf_{\scenario\in\Omega_n}\max_{\u\in\UU}\inf_{\w\in\WW}\left\{ \min\left\{ \bigwedge_{k=n+1}^{N}\Phi_k^\threshold(\state_k,\control_k),\Theta^\threshold(\state_{N+1})\right\}\middle|\  \begin{matrix}\state_{k+1}=\dynamics_k(\state_k,\control_k,\scenario_k)\\  k\in\titf{n+1}{N},\\\state_{n+1}=\dynamics_n(\ic,\bar\control_n,\scenario)\end{matrix}
\right\}\\
&=\inf_{\scenario\in\Omega_n}V^\threshold_{n+1}\left(\dynamics_n(\ic,\bar\control_n,\scenario)\right).\end{aligned}
\end{equation*}

%\begin{equation*}
%\begin{aligned}
%&\inf_{\w\in\WW}\left\{\min\left\{ \bigwedge_{k=n+1}^{N}\Phi_k^\threshold(\state_k,\bar\control_k),\Theta^\threshold(\state_{N+1})\right\}\middle|\  \begin{matrix}\state_{k+1}=\dynamics_k(\state_k,\bar\control_k,\scenario_k)\\  k\in\titf{n+1}{N},\\\state_{n+1}=\dynamics_n(\ic,\bar\control_n,\scenario_n)\end{matrix}
%\right\}\\
%&\leq\max_{\u\in\UU}\inf_{\w\in\WW}\left\{ \min\left\{ \bigwedge_{k=n+1}^{N}\Phi_k^\threshold(\state_k,\control_k),\Theta^\threshold(\state_{N+1})\right\}\middle|\  \begin{matrix}\state_{k+1}=\dynamics_k(\state_k,\control_k,\scenario_k)\\  k\in\titf{n+1}{N},\\\state_{n+1}=\dynamics_n(\ic,\control,\scenario_n)\end{matrix}
%\right\}\\
%&=\max_{\u\in\UU}\inf_{\scenario\in\Omega_n}\inf_{\w\in\WW}\left\{ \min\left\{ \bigwedge_{k=n+1}^{N}\Phi_k^\threshold(\state_k,\control_k),\Theta^\threshold(\state_{N+1})\right\}\middle|\  \begin{matrix}\state_{k+1}=\dynamics_k(\state_k,\control_k,\scenario_k)\\  k\in\titf{n+1}{N},\\\state_{n+1}=\dynamics_n(\ic,\control,\scenario)\end{matrix}
%\right\}\\
%&\leq 
%	\inf_{\scenario\in\Omega_n}\max_{\u\in\UU}\inf_{\w\in\WW}\left\{ \min\left\{ \bigwedge_{k=n+1}^{N}\Phi_k^\threshold(\state_k,\control_k),\Theta^\threshold(\state_{N+1})\right\}\middle|\  \begin{matrix}\state_{k+1}=\dynamics_k(\state_k,\control_k,\scenario_k)\\  k\in\titf{n+1}{N},\\\state_{n+1}=\dynamics_n(\ic,\control,\scenario)\end{matrix}
%\right\}\\
%&=\inf_{\scenario\in\Omega_n}V^\threshold_{n+1}\left(\dynamics_n(\ic,\control,\scenario)\right).\end{aligned}
%\end{equation*}

Whence we get $V^\threshold_n(\ic)\leq W^\threshold_n(\ic)$.

Now, to prove the other inequality, let us point out that, since the maximum is attained in the definition of any of the functions $W^\threshold_k(\cdot)$, there is an optimal feedback control $\feedback_k:\X\to\U$ such that
$$W^\threshold_k(\state):=\min\left\{\inf_{\scenario_k\in\Omega_k}W^\threshold_{k+1}\left(\dynamics_k(\state,\feedback_k(\state),\scenario_k\right),\Phi_k^\threshold(\state,\feedback_k(\state))\right\},\qquad\forall \state\in\X,$$
where we have set $W^\threshold_{N+1}\defegal V^\threshold_{N+1}$. The use of $W^\threshold_{k+1}$ instead of $V^\threshold_{k+1}$ in the preceding equality is a consequence of the induction hypothesis.

 Take $\w\in\WW$ arbitrary and define 
$$\state_{k+1}=\dynamics_k(\state_k,\feedback_k(\state_k),\scenario_k),\qquad\forall k\in\titf{n}{N},\qquad\state_{n}=\ic.$$
It follows then that
$$W^\threshold_k(\state_k)\leq\min\left\{W^\threshold_{k+1}\left(\state_{k+1}\right),\Phi_k^\threshold(\state,\feedback_k(\state_k))\right\},\qquad\forall k\in\titf{n}{N}.$$
Therefore, using this inequality repeatedly we get
$$W^\threshold_n(\ic)=W^\threshold_n(\state_n)\leq \min\left\{\bigwedge_{k=n}^{N}\Phi_k^\threshold(\state_k,\feedback_k(\state_k)),\Theta^\threshold(\state_{N+1})\right\}.$$
Since this is true for any $\w\in\WW$, it yields to
$$W^\threshold_n(\ic)\leq\inf_{\w\in\WW}\left\{ \min\left\{\bigwedge_{k=n}^{N}\Phi_k^\threshold(\state_k,\feedback_k(\state_k)),\Theta^\threshold(\state_{N+1})\right\}~\middle|~\ \begin{matrix}\state_{k+1}=\dynamics_k(\state_k,\feedback_k(\state_k),\scenario_k)\\ k\in\titf{n}{N},\ \state_{n}=\ic\end{matrix}
\right\}$$
%
%$$W^\threshold_n(\ic):=\max_{u\in\U}\min\left\{\inf_{\scenario\in\Omega_n}V^\threshold_{n+1}\left(\dynamics_k(\ic,\control,\scenario\right),\Phi_k^\threshold(\ic,\control)\right\}$$
%
%$$V^\threshold_n(\ic)\defegal\max_{\u\in\UU}\inf_{\w\in\WW}\left\{ \min\left\{\bigwedge_{k=n}^{N}\Phi_k^\threshold(\state_k,\control_k),\Theta^\threshold(\state_{N+1})\right\}~\middle|~\ \begin{matrix}\state_{k+1}=\dynamics_k(\state_k,\control_k,\scenario_k),\\ k\in\titf{n}{N},\ \state_{n}=\ic\end{matrix}
%\right\}.$$	
Finally, by taking supremum over $\u\in\UU$, we get that $W^\threshold_n(\ic)\leq V^\threshold_n(\ic)$. Thus , by induction, the conclusion follows.

\end{proof}

\subsection{A scheme for computing the weak Pareto front}\label{schemewpf}

To summarize, by combining Theorem \ref{thm:weak_front} and Proposition \ref{prop:dpp}, we obtain a practical method (Algorithm \ref{algo2}) to compute the weak Pareto front of the set of robust sustainable thresholds associated with a control system with  constraints.

\new{For implementing this algorithm, it is necessary to define two meshes  $\X_h \subset \X$ and $S_h\subseteq\RR^m$ of size $0<h\ll1$  as  computational domains (state and thresholds). Then for any $n\in\titf{0}{N}$,  the   function $V^\threshold_n(\cdot)$ in Proposition \ref{prop:dpp} has to be computed for every $\ic' \in \X_h$ reachable in $n$  steps from $\ic$, and for all $\threshold \in S_h$, a procedure that  could be too expensive, which is not surprising because the method is based on the dynamic programming principle. Nevertheless, from Theorem \ref{thm:weak_front} and Proposition \ref{prop:dpp}, the method introduced in Algorithm \ref{algo2} will not need a large mesh $S_h$, as it is explained in the example shown in Section \ref{sec:simulations}. }

\begin{algorithm}\caption{Computing the weak Pareto front}\label{algo2}
	\KwIn{$\ic\in\X$, $N\in\NN$, $\dynamics:\titf{0}{N}\times\X\times\U\times \Omega\to\X$, $g:\titf{0}{N}\times\X\times\U\to\RR^m$}
	Let $\X_h \subset \X$ and $S_h\subseteq\RR^m$  be two meshes of size $0<h\ll1$ for  computational domains (state and thresholds).
	
	For $n \in  \titf{0}{N}$ let $\X_h^n \subset \X_h$ be the set of points in $\X_h$ reachable from $\ic$ in $n$ steps.
	
	Let $\TT$ and $\wpareto$ two empty arrays.
	
	\For{ $\threshold_i \in S_h$}{
	
		\For{ $\ic' \in \X_h^N$}{
		Compute $V_N^{\threshold_i}(\ic')=\underset{\control\in\U}\max \; \Phi^{\threshold_i}_N(\ic',\control).$\label{algo:step1}}
		
		}
		
		Set  $n=N-1$.
		
\While {$n\geq 0$}{
	
		\For{ $\threshold_i \in S_h$}{
		\For{ $\ic' \in \X_h^n$}{
		\hspace{-0.2cm}Compute $V_n^{\threshold_i}(\ic')=\underset{u\in\U}\max\min\left\{\underset{\scenario\in \Omega_n}\min V_{n+1}^{\threshold_i}\left(\dynamics_n(\ic',u,\scenario)\right),\Phi^{\threshold_i}_n(\ic',u)\right\}.$\label{algo:step2}}
		
		}
		
		Set $n= n-1$.}
		\For{ $\threshold_i \in S_h$}{Save $\threshold_i+V_0^{\threshold_i}(\ic)\bfone$ in $\wpareto$}

	\Return{$\wpareto$ and $\TT=\wpareto + \RR^m_-$}
	\end{algorithm}

\if{
\begin{algorithm}\caption{Computing the weak Pareto front}\label{algo2}
	\KwIn{$\ic\in\X$, $N\in\NN$, $\dynamics:\titf{0}{N}\times\X\times\U\to\X$, $g:\titf{0}{N}\times\X\times\U\to\RR^m$}
	Let $S_h\subseteq\RR^m$  be a mesh of size $0<h\ll1$ for a computational domain.
	
	Let $\TT_N, \TT_{N-1},\ldots,\TT_0$ and $\wpareto^N, \wpareto^{N-1},\ldots,\wpareto^0$ be empty arrays.
	
	\For{ $\threshold_i \in S_h$}{
	
		Compute $V_N^{\threshold_i}(\ic)=\max_{\control\in\U}\Phi^{\threshold_i}(N,\ic,\control).$\label{algo:step1}
		
		\If{ $V_N^{\threshold_i}(\ic)\geq 0$}{Save $\threshold_i$ in $\TT_N$.}
		\Else{Save $\threshold_i+V_N^{\threshold_i}(\ic)\bfone$ in $\wpareto^N$}}
		
		Set $S_h=\TT_N$ and $n=N-1$.
		
\While {$S_h\neq\emptyset$ and $n\geq 0$}{
	
		\For{ $\threshold_i \in S_h$}{
		
		Compute $V_n^{\threshold_i}(\ic)=\max_{u\in\U}\min\left\{ \min_{\scenario\in\Omega_n} V_{n+1}^{\threshold_i}\left(\dynamics_n(\ic,\control,\scenario)\right),\Phi^{\threshold_i}(n,\ic,\control,\scenario)\right\}.$\label{algo:step2}
		
		\If{ $V_n^{\threshold_i}(\ic)\geq 0$}{Save $\threshold_i$ in $\TT_n$.}
		\Else{Save $\threshold_i+V_n^{\threshold_i}(\ic)\bfone$ in $\wpareto^n$}}
		
		Set $S_h=\TT_n$ and $n\leftarrow n-1$.}

	\Return{$\TT_N, \TT_{N-1},\ldots,\TT_0$ and $\wpareto^N, \wpareto^{N-1},\ldots,\wpareto^0$}
	\end{algorithm}
}\fi

\section{Simulations}\label{sec:simulations}

In this section, we illustrate the computation of the set of robust sustainable thresholds $\TT(\ic)$, for one example based on renewable resource management, inspired by \cite{Clark:1990}. In the example, the stock of a renewable resource in period $k$ is represented by $\state_k \ge 0$, and its dynamics with harvesting (or catch) $\control_k$ are described by
$$\state_{k+1} = \dynamics(\state_k,\control_k, \scenario_k)= f(\state_k,\scenario_k) -\control_k$$
where $f$ stands for the renewable function of the stock, depending on the scenario $\scenario_k \in \Omega_k \equiv \Omega = \{\scenario_a,\scenario_b\}$, for all $k \in \titf{0}{N}$. 

For the above control system, suppose that a social planner has the objective of ensuring both, minimal resource stocks in the nature and minimal harvesting. The first requirement is associated to the sustainability of the resource and the second to economical, social or food security issues. Reformulated from a viability viewpoint, the problem relates to sustaining both stock and harvest through the thresholds $\state^{\lim}$ and $h^{\lim}$ as follows:
\begin{equation}\label{eq:ex1}
\left\{\begin{array}{l}
\state_{k+1}=f(\state_{k},\scenario_k) - \control_k, \; \\
\state_0 =\ic ~\mbox{ given (the current state of the resource)} \\
\state_k \ge \state^{\lim}  
\\
\control_k \ge h^{\lim} . \end{array}\right.
\end{equation}

\if{
The computation of the set $\TT(\ic)$ corresponds to the identification of viable thresholds $\state^{\lim}$  and $h^{\lim}$ with respect to current state $\ic$. 

We will see that the maximal sustainable yield (MSY) level is a tipping point in the determination of the sustainable thresholds. To introduce this concept, we denote by $\sigma(\state)$ the harvest level obtained for the equilibrium biomass level $\state$; that is,
 $$\sigma(\state)=f(\state)-\state.$$ 
 Therefore, the MSY is the catch level at equilibrium for which this quantity is maximized; that is,
$${\msy} = \max_{\state \ge 0} \sigma(\state).$$
}\fi

We shall study this very simple example, because when the dynamics $f(\cdot,\scenario)$ is nondecreasing for  $\scenario \in \Omega$, we can analytically compute the set of { robust sustainable thresholds when the horizon is infinity, and then we are able to compare this analytical expression with the result given by our method for computing  $\TT(\ic)$ (Algorithm \ref{algo2}). A first interesting property we can show in this framework is presented in the following remark.

\begin{remark}\label{rem:equality}
Under the assumption that the function $f(\cdot,\scenario)$ in  \eqref{eq:ex1} is nondecreasing, for  $\scenario \in \Omega$,  one has  
$$\TT(\ic)=\hat \TT(\ic)= \bigcap_{\scenario \in \Omega} \TT^\scenario(\ic)$$
 where 
$\TT^\scenario(\ic)$ is the set of sustainable thresholds for the (deterministic) system 
\begin{equation}\label{eq:ex1det}
\left\{\begin{array}{l}
\state_{k+1}=f(\state_{k},\scenario) - \control_k, \; \\
\state_0 =\ic \\
\state_k \ge \state^{\lim}  
\\
\control_k \ge h^{\lim}  \end{array}\right.
\end{equation}
associated to the constant scenario $\w_\scenario=(\scenario_k)^N_{k=0}$ such that $\scenario_k=\scenario$ for all $k\in\titf{0}{N}$. Thanks to Remark \ref{rem:linkdeterministic}, to prove this claim we only need to show that $\hat \TT(\ic)\subseteq \TT(\ic)$. For this purpose, let us consider a threshold $\threshold=(\state^{\lim},h^{\lim})\in \hat \TT(\ic)$. It is then immediate to verify that  the set of admissible controls $\UU_\ic^{\w_\scenario}(\threshold)$ is nonempty for all $\scenario\in\Omega$. Let $\u^\scenario=({\control}^\scenario_k)^N_{k=0}$ be an element of $\UU_\ic^{\w_\scenario}(\threshold)$, for $\scenario\in\Omega$, and $\bar{\u}=(\bar{\control}_k)^N_{k=0}\in \UU$ be defined by $\bar{\control}_k:=\inf_{\scenario\in\Omega} {\control}^\scenario_k$, for $k\in\titf{0}{N}$. Being $\u^\scenario$ admissible for all $\scenario\in\Omega$, one has
$$
\bar{\control}_k \geq h^{\lim},\qquad\forall k\in\titf{0}{N}.
$$
Furthermore, thanks to the monotonicity of $f(\cdot,\scenario)$ and the definition of $\bar\u$, it can be easily shown that for any $\w=(\scenario_k)^N_{k=0}\in\WW$, the trajectory 
$\x^\w_\ic(\bar\u)=(\state_k)^{N+1}_{k=0}$ satisfies
$$
\state_k\geq \state^{\scenario_k}_k,\qquad\forall k\in\titf{0}{N},
$$
where we denoted by $\x^{\w_\scenario}_\ic(\u^\scenario)=(\state^{\scenario}_k)^{N+1}_{k=0}$ the admissible (deterministic) trajectory associated to $\scenario\in \Omega$. From the last inequality  and the admissibility of the control $\u^\scenario$ we obtain
$$
\state_k \geq \state^{\lim},\qquad\forall k\in\titf{0}{N},
$$
and then we can conclude that $\bar\u\in \UU^\w_\ic(\threshold)$. Thanks to the arbitrariness of $\w\in\WW$ the latter implies that $\bar\u\in \bigcap_{\w\in\WW} \UU^\w_\ic(\threshold)$, that is   $\threshold\in \TT(\ic)$.
\end{remark}

The previous remark links the characterization of the set of robust sustainable thresholds $\TT(\ic)$ to the one of its deterministic versions $\{\TT^\scenario(\ic),\,\scenario\in \Omega\}$ obtained for constant scenarios.
Interestingly, as already pointed out in \cite{GajHer19}, for infinite horizon problems, i.e. if $N=\infty$, the monotonicity of $f$ allows an analytical computation of $\TT^\scenario_\infty$ and then, in virtue of Remark \ref{rem:equality} of $\TT_\infty(\ic)$.
} 

In particular,  let us consider the Beverton-Holt population dynamics
\begin{equation}\label{eq:BH}
f(\state,\scenario)= (1+r(\scenario))\state\left(1+\frac{r(\scenario)}{K(\scenario)}\state\right)^{-1}
\end{equation}
where the intrinsic growth $r(w) \in \{r(\scenario_a),r(\scenario_b)\}$ and carrying capacity $K(\scenario) \in \{K(\scenario_a),K(\scenario_b)\},$ are positive parameters depending on the scenario $\scenario \in \Omega$. 
For this Beverton-Holt growth function \eqref{eq:BH} and a fixed scenario $\scenario$, the maximal sustainable yield (MSY) level is a tipping point in the determination of the sustainable thresholds of the deterministic system \eqref{eq:ex1det}, and it is attained at the biomass level $\state_{\msy}^\scenario$ given by
$$\state_{\msy}^\scenario= \frac{K(\scenario)}{1+\sqrt{1+r(\scenario)}}.$$
When the horizon is infinity ($N=+\infty$), the viability kernel (associated to \eqref{eq:ex1det} for the scenario $\scenario$ fixed) has been calculated analytically in \cite{DLD}, and as it is shown in \cite{DG2018,GajHer19} the (deterministic) set of sustainable thresholds associated to \eqref{eq:ex1det} is given by
\begin{equation}\label{eq:sinftydet}
\TT^\scenario_\infty(\ic) = \{(\state^{\lim},h^{\lim}) ~|~   \state^{\lim} \leq \min\{\state_0,K(\scenario)\}; ~~  h^{\lim} \leq \sigma_\scenario(\state^{\lim}) \},
\end{equation}
where the function $\sigma_\scenario(\cdot)$ is defined by 
 \begin{equation}\label{eq:sigma}
 \sigma_\scenario(\state)=f(\state,\scenario)-\state,
 \end{equation}
 and represents the harvesting (or yield) at equilibrium when the steady state stock is $\state$, for a fixed scenario $\scenario$. 

Since $\TT(\ic)$ {coincides} with $\hat \TT(\ic) = \bigcap_{\scenario \in \Omega} \TT^\scenario(\ic)$  and $\TT^\scenario(\ic)$ approaches to $\TT^\scenario_\infty(\ic)$ when $N \to \infty$, 
the objective of this example is to compare $\TT(\ic)$ -computed by our method for $N$ large enough- with respect to  $\hat \TT_\infty(\ic) = \bigcap_{\scenario \in \Omega} \TT^\scenario_\infty(\ic)$ computed analytically using \eqref{eq:sinftydet}.

\if{
\begin{remark}\label{rem:oposite}
Note that the constraints $\state_k \ge \state^{\lim} $ and $\control_k \ge h^{\lim}$ in \eqref{eq:ex1} are given in the opposite sense of the constraints established in \eqref{eq:constraints} for the general formulation. Instead of reformulating these constraints as $-\state_k \geq -\state^{\lim} $ and $-\control_k \geq -h^{\lim}$ to fit the general formulation, we deal directly with the original constraints; therefore, the definition of $\TT(\ic)$ changes slightly. Instead of $\TT(\ic) + \RR^m_- = \TT(\ic)$, dealing with the original constraints will result in $\TT(\ic) + \RR^m_- = \TT(\ic)$. On the other hand, for the problem studied in this example, we have no interest in the negative thresholds $\state^{\lim} $ and $h^{\lim}$. For this reason, in Figure \ref{fig:example1}, we only depict the sustainable thresholds in the positive \pg{orthant}.
\end{remark}
}\fi

\begin{figure}[!htp]
%\begin{center}
~\hspace{-2.4cm}\includegraphics[scale=.85,clip]{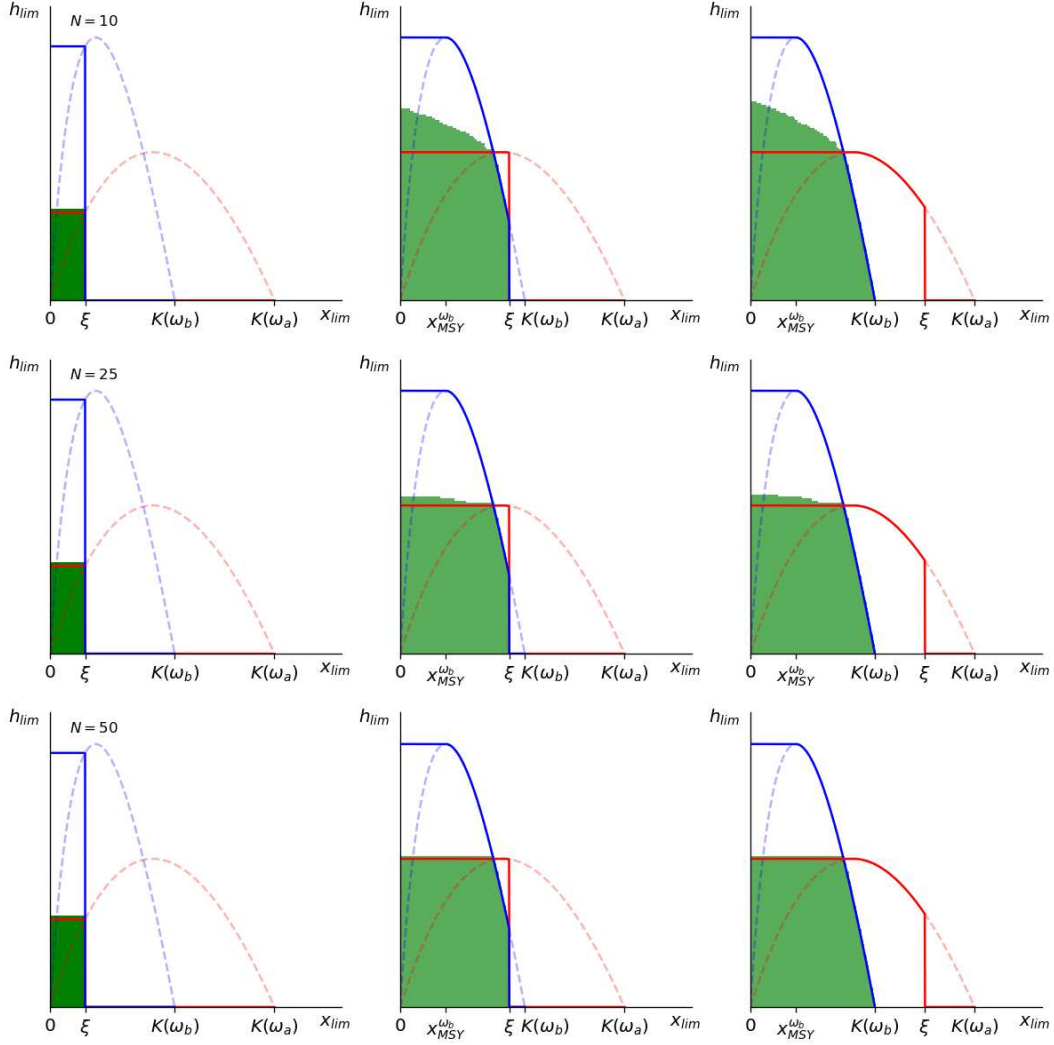}
 \caption{Set of robust sustainable thresholds $ \TT(\ic)$ (green) for different time horizons {$N=10$ (first row), $N=25$ (second row), and $N= 50$ }(third row) and  three initial stocks  $\ic$ (displayed in each column). The red and blue curves in each panel correspond to the weak Pareto fronts of sets $\TT^{\scenario_a}_\infty(\ic)$ and $\TT^{\scenario_b}_\infty(\ic)$  respectively, from where one identifies  the set $\hat \TT_\infty(\ic)=  \TT^{\scenario_a}_\infty(\ic) \cap \TT^{\scenario_b}_\infty(\ic)$. {The dashed curves represent functions $\sigma_{\scenario_a}$ and $\sigma_{\scenario_a}$ defined in \eqref{eq:sigma} used for computing  $\TT^{\scenario_a}_\infty(\ic)$ and $\TT^{\scenario_b}_\infty(\ic)$ (see \eqref{eq:sinftydet})}. The parameters for the resource  dynamics were set to $r(\scenario_a)=0.39$,  $r(\scenario_b)=2$, $K(\scenario_a)=90$, and $K(\scenario_b)=50$.}\label{fig:example1}
%\end{center}
\end{figure}

In Figure \ref{fig:example1} we show the set of robust sustainable thresholds $ \TT(\ic)$ considering different time horizons {$N=10$ (first row), $N=25$ (second row), and $N= 50$ (third row)} for three initial endowments of the resource $\ic$ (displayed in each column).  Also, through \eqref{eq:sinftydet}, we compute analytically the sets of sustainable thresholds associated to the deterministic system \eqref{eq:ex1det} for the constant scenarios $\scenario \in \Omega = \{\scenario_a,\scenario_b\}$ and infinite horizon ($N=+\infty$). We illustrate these sets by depicting the weak Pareto fronts of  $\TT^{\scenario_a}_\infty(\ic)$ (red) and $\TT^{\scenario_b}_\infty(\ic)$ (blue), from where it is easy to identify the set 
$$\hat \TT_\infty(\ic) = \bigcap_{\scenario \in \Omega} \TT_\infty^\scenario(\ic)=  \TT^{\scenario_a}_\infty(\ic) \cap \TT^{\scenario_b}_\infty(\ic) .$$

\if{ 
$\TT_\infty(\ic)$ in the infinite horizon case, computed analytically (first row), and the numerical approximation of $\TT(\ic)$ when $N=20$ (second row) for initial conditions $\ic$ in the following cases: (i) $0< \ic < \state_{\msy}$ (first column); (ii) $\state_{\msy} <  \ic < K$ (second column); and (iii) $K <  \ic$ (third column). }\fi

The procedure for obtaining the set of robust sustainable thresholds $ \TT(\ic)$ was conducted by computing the weak Pareto front $\pareto^\mathcal{W}\left(\TT(\ic)\right)$ and then using the equality (see Remark \ref{rem:charact})
$$ \TT(\ic)= \pareto^\mathcal{W}\left(\TT(\ic)\right)+ \RR^m_- .$$
The Pareto front $\pareto^\mathcal{W}\left(\TT(\ic)\right)$ is computed using the elements and results presented in Section \ref{weakpf}, specifically with the method outlined in Algorithm \ref{algo2}.

In more details, in the positive \pg{orthant} of $\RR^2$ (space of thresholds)  we consider the mesh 
\begin{equation}\label{eq:mesh}
S_d=\{(jd,\bar h^{\lim}) ~|~j=0,1,\ldots,N_d\} \cup \{(\bar \state^{\lim},j d) ~|~j=0,1,\ldots,N_d\},
\end{equation}
with $0<d\ll1$ as the size of the mesh, $N_d \in \NN$, and $\bar \state^{\lim}$, $\bar h^{\lim} > 0$ large enough. For each vector of thresholds $ \threshold = (\state^{\lim},h^{\lim})$ in the mesh $S_d$, we compute $\aux{\ic}{\threshold}$ defined in \eqref{eq:Wc}. Taking   $\bar \state^{\lim}$ and  $\bar h^{\lim}$ sufficiently large ensures that vectors $\threshold$ in the mesh are not in $\TT(\ic)$. Hence, from Theorem \ref{thm:weak_front}, we obtain $\aux{\ic}{\threshold} < 0$ and  we find that $p(\threshold)\defegal \threshold + \aux{\ic}{\threshold} \bfone$ is in the weak Pareto front for all vector of thresholds $ \threshold$ in $S_d$. Thus, we obtain the weak Pareto front of $\TT(\ic)$ and, a fortiori, the entire set $\TT(\ic)$.

Since  
$$\TT(\ic)  = \hat \TT(\ic) = \TT^{\scenario_a}(\ic) \cap \TT^{\scenario_b}(\ic) $$
(see Remark  \ref{rem:equality}) and due to  the set $\hat \TT(\ic) $ approaching $\hat \TT_\infty(\ic)$ when the time horizon  $N$ increases, for $N$ large, as in the third row of Figure \ref{fig:example1}, we should obtain
$$\TT(\ic){\; \overset{N\to\infty}\longrightarrow} \; \hat \TT_\infty(\ic) = \TT^{\scenario_a}_\infty(\ic) \cap \TT^{\scenario_b}_\infty(\ic)$$
{as the numerical tests reported in Figure \ref{fig:example1}  confirm.}

%\begin{figure}[ht]
%\begin{center}
%\hspace{-.655cm}\includegraphics[scale=0.09]{figure_SST.eps}
% \caption{set of robust sustainable thresholds \pg{$\TT_\infty(\ic)$  (first row) and their approximations given by $\TT(\ic)$ (second row) }with $N=20$ for different initial conditions $\ic$: first column: $\ic < x_{\msy}$; second column: $x_{\msy} < \ic< K$; third column: $\ic > K$. The parameters for the stock dynamics are set to $r= 1.75$ and $K=50$.}\label{fig:example1}
%\end{center}
%\end{figure}

\if{
\new{For obtaining Figure \ref{fig:example1}, the mesh $S_h$ in Algorithm \ref{algo2} given by \eqref{eq:mesh}, we have taken  $N_h=80$ (i.e, 160 vectors of thresholds). The same number of points was considered in the mesh $\X_h$, the discretization of the state space. The algorithm was implemented in Python (Jupyter notebook) and the CPU time for the horizon $N=20$ was 240 seconds.  }
}\fi

%\newpage
%\bibliographystyle{plain}
\bibliographystyle{siamplain}
\bibliography{biblio1} 
\end{document}